\documentclass[reqno]{amsart}
\usepackage{amssymb,amsthm,amsfonts,amstext}
\usepackage{amsmath}

\makeatletter \@addtoreset{equation}{section} \makeatother

\renewcommand\thefigure{\thesection.\@arabic\c@figure}
\renewcommand\thetable{\thesection.\@arabic\c@table}
\newtheorem{theorem}{Theorem}[section]
\newtheorem{lemma}[theorem]{Lemma}
\newtheorem{proposition}[theorem]{Proposition}
\newtheorem{corollary}[theorem]{Corollary}

\newtheorem{definition}[theorem]{Definition}

\newcommand{\mc}[1]{{\mathcal #1}}
\newcommand{\mf}[1]{{\mathfrak #1}}
\newcommand{\mb}[1]{{\mathbf #1}}
\newcommand{\bb}[1]{{\mathbb #1}}

\newcommand{\<}{\langle}
\renewcommand{\>}{\rangle}

\renewcommand{\ll}{\langle\!\langle\,}
\renewcommand{\gg}{\,\rangle\!\rangle}

\renewcommand{\Cap}{{\rm cap}}

\begin{document}

\title{Quenched scaling limits of trap models}

\author{M. Jara, C. Landim, A. Teixeira}

\begin{abstract}
  Fix a strictly positive measure $W$ on the $d$-dimensional torus
  $\bb T^d$. For an integer $N\ge 1$, denote by $W^N_x$, $x=(x_1,
  \dots, x_d)$, $0\le x_i <N$, the $W$-measure of the cube $[x/N,
  (x+\mb 1)/N)$, where $\mb 1$ is the vector with all components equal
  to $1$. In dimension $1$, we prove that the hydrodynamic behavior of
  a superposition of independent random walks, in which a particle
  jumps from $x/N$ to one of its neighbors at rate $(N W^N_x)^{-1}$,
  is described in the diffusive scaling by the linear differential
  equation $\partial _t \rho = (d/dW)(d/dx) \rho$. In dimension $d>1$,
  if $W$ is a finite discrete measure, $W=\sum_{i\ge 1} w_i
  \delta_{x_i}$, we prove that the random walk which jumps from $x/N$
  uniformly to one of its neighbors at rate $(W^N_x)^{-1}$ has a
  metastable behavior, as defined in \cite{bl1}, described by the
  $K$-process introduced in \cite{fm1}.
\end{abstract} 

\address{Ceremade, UMR CNRS 7534,
    Universit\'e de Paris IX - Dauphine,
    Place du Mar\'echal De Lattre De Tassigny
    75775 Paris Cedex 16 - France.
\newline
e-mail: \rm \texttt{jara@ceremade.dauphine.fr}}

\address{\noindent IMPA, Estrada Dona Castorina 110, CEP 22460 Rio de
  Janeiro, Brasil and CNRS UMR 6085, Universit\'e de Rouen, Avenue de
  l'Universit\'e, BP.12, Technop\^ole du Madril\-let, F76801
  Saint-\'Etienne-du-Rouvray, France.  \newline e-mail: \rm
  \texttt{landim@impa.br} }

\address{\noindent Departement Mathematik ETH-Z\"urich, HG G 47.2,
  R\"amistrasse 101, 8092 Z\"urich, Switzerland.
\newline e-mail: \rm \texttt{augusto.teixeira@math.ethz.ch} }

\keywords{Trap models, scaling limit, hydrodynamic equation, gap
  diffusions, metastability} 

\maketitle

\section{Introduction}
\label{sec0}

Scaling limits of randoms walks in random trap environments have been
examined recently \cite{fin, bc1, bc3} as stochastic models which
exhibit aging \cite{fin, bc4, bcm1}, a phenomenon of considerable
interest in physics and mathematics.

To describe the dynamics, fix an unoriented graph $G = (V,E)$ with
finite degree and consider a sequence of i.i.d.\! strictly positive
random variables $\{\xi_z : z\in V\}$ indexed by the vertices. Let
$\{X_t : t\ge 0\}$ be a continuous time random walk on $V$ which waits
a mean $\xi_z$ exponential time at site $z$, at the end of which it
jumps to one of its neighbors with uniform probability.

The time spent by the random walk on a vertex $z$ is proportional to
the value of $\xi_z$. It is thus natural to regard the environment as
a landscape of valleys or traps with depth given by the value of the
random variables $\{\xi_z : z\in V\}$.  As the random walk evolves, it
explores the random landscape, finding deeper and deeper traps, and
aging appears as a consequence of the longer and longer times the
process remains at the same vertex.

It is clear from the description that random walks on random trap
environments should present a very rich scaling fractal structure if
one chooses appropriate graphs and random environments. For each given
time scale, only traps at a certain depth matter. The deeper valleys
are too sparse to influence the evolution and the shallower wells are
not deep enough to retain the process.

We are concerned in this article with the lattice case: $\{\xi_z :
z\in \bb Z^d\}$ is a sequence of i.i.d.\! strictly positive random
variables and $\{X_t : t\ge 0\}$ a continuous time random walk on
$\bb Z^d$ which waits a mean $\xi_z$ exponential time at site $z$, at
the end of which it jumps to one of its neighbors with probability
$1/2d$.

When $\xi_0$ has finite mean, for almost all environments $\{\xi_z :
z\in \bb Z^d\}$, the rescaled random walk $\epsilon
X_{t\epsilon^{-2}}$ converges in distribution to a Brownian motion.
In dimension $1$, we can use the method of random time change to study
the problem explicitly and a simple computation establishes the result
\cite{s}. In this case, the diffusion coefficient is equal to $
E[\xi_0]^{-1}$, the harmonic mean of the random rates $\{\xi_z^{-1} :
z\in \bb Z^d\}$. Observing that the random walk is a martingale, in
higher dimension, by examining the evolution of the environment as
seen from the position of the random walk, the proof of the invariance
principle is reduced to the proof of an ergodic theorem for the
dynamics of the environment \cite{pv1}. An explicit formula for the
variance is, however, no longer available.

To investigate the case where the environment has an infinite mean, a
natural assumption is to suppose that the distribution of $\xi_0$
belongs to the domain of attraction of an $\alpha$-stable law,
$0<\alpha<1$. The variables $\{\xi_z : z\in \bb Z^d\}$ take now large
values in certain sites, forcing the random walk to stay still for a
long time when it reaches one of them, causing a macroscopic
subdiffusive behavior.

In dimension $1$, Fontes, Isopi and Newman \cite{fin} proved under
these hypotheses that for almost all environments, the random walk
converges, in the time scale $t^{1+(1/\alpha)}$, to a singular
diffusion with a random discrete speed measure. In dimension $d\ge 2$,
Ben Arous and ${\rm \check C}$ern\'y \cite{bc1} proved that for almost
all environments the Bouchaud trap model converges in a proper time
scale, $t^{2/\alpha}$ in dimension $d\ge 3$ and a scale logarithmic
smaller than $t^{2/\alpha}$ in dimension $2$, to the
fractional-kinetic process, a self-similar, non-Markovian, continuous
process, obtained as the time change of a Brownian motion by the
inverse of an independent $\alpha$-stable subordinator. In fact, they
proved, under quite general conditions on the environment, that the
clock process converges to an $\alpha$-stable subordinator, for a
large range of time scales \cite{bc3}. In these time scales, the
random walk does not visit the deepest traps, but exhibit an aging
behavior. During the exploration of the random scenery, the process
discovers deeper and deeper traps which slow down its evolution, the
mechanism responsible for the aging phenomenon.  We refer to
\cite{bc2, c1} for recent reviews.

We present in this article two results. The first one establishes the
hydrodynamic behavior, almost sure with respect to the environment, of
a superposition of independent random walks evolving on the
one-dimensional torus with a trap environment of $\alpha$-stable
i.i.d.\! random variables. The hydrodynamic equation, describing the
macroscopic evolution of the density, is given by the generalized
second order linear equation
\begin{equation*}
\frac d{dt} \rho (t,x) \; =\; \frac {d}{dW}\, \frac{d}{dx} \, \rho
(t,x) \; ,
\end{equation*}
where $W$ is an $\alpha$-stable subordinator deriving from the
realization of the environment. The Krein--Feller operator $(d/dW)
(d/dx)$ is the generator of the singular diffusion obtained by
Fontes, Isopi and Newman \cite{fin} as scaling limit of the random
walk in the trap environment.

The striking feature of this result is that the random environment
survives entirely in the limit, since even the differential operator,
which describes the macroscopic evolution of the density, depends on
the specific realization of the environment. A similar phenomenon was
observed in \cite{fjl1, fl1, v} for exclusion processes with
$\alpha$-stable random conductances.

The second result describes the evolution of the random walk in the
random environment, produced by $\alpha$-stable i.i.d.\! random
variables, in dimension $d\ge 2$ in the time scale needed to visit the
deepest traps. In the notation of Theorem 4.1 in \cite{bc3}, this
corresponds to the case $\gamma = 0$.

In dimension $2$, on the time scale $N^{2/\alpha} \log N$, we prove
that the random walk, evolving on the discrete torus $(\bb Z/N \bb
Z)^2$, converges to the Markov $K$-process introduced by Fontes and
Mathieu \cite{fm1}, which in the present context can be informally
described as follows. The state space is formed by the countable and
dense subset of deepest traps. The process stays at one of these
sites an exponential time, with expectation proportional to the depth
of the trap, at the end of which it jumps to a new location, chosen
with uniform probability among the deepest traps. The scaling limit is
similar in dimension $d\ge 3$, but the time scale is now
$N^{d/\alpha}$.  In the terminology of \cite{bl1}, these results
establish the metastability of the random walk in dimension $d\ge 2$.

Convergence to the $K$-process has been proved by Fontes and Mathieu
\cite{fm1} for the trap model in the complete graph and by Fontes and
Lima \cite{fli1} for the trap model in the hypercube. We believe that
this is a universal behavior of random walks on graphs with heavy
tailed random trap environments in the ergodic time scale, the scale
proportional to the time needed to jump from one very deep trap to
another.  At least in sufficiently high dimension.

It is in fact quite surprising that even in low dimensions the
geometry of the torus is completely wiped out in the scaling limit of
the random walk in a random trap environment, as proved below.

We conclude this introduction by specifying the random environment we
consider in this article. Though we shall work on the torus, we
present the construction on $\bb R^d$.  Let $\lambda$ be the measure
on $\bb R^d \times (0,\infty)$ given by $\lambda = \alpha
w^{-(1+\alpha)} dx\, dw$, $0<\alpha<1$. Denote by $\{(\mb x_i,w_i) \in
\bb R^d \times (0,\infty): i\ge 1\}$ the marks of a Poisson point
process of intensity $\lambda$, and define the measure $W$ on $\bb
R^d$ by
\begin{equation*}
W \;=\; \sum_{i\ge 1} w_i \, \delta_{\mb x_i}\;.
\end{equation*}
For $z =(z_1, \dots , z_d)$ in $\bb Z^d$, let $[z/N, [z+ \mb 1]/N)$ be the
$d$-dimensional cube $\prod_{1\le i\le d}$ $[z_i/N, [z_i+1]/N)$ and let
\begin{equation*}
\xi_z^N = N^{d/\alpha} \sum_{i \ge 1}   w_i \,
\mb 1\{ \mb x_i \in [z/N , (z+\mb 1)/N) \} \; ,
\end{equation*}
where $\mb 1\{A\}$ stands for the indicator of the set $A$. We show in
the next section that, for each $N\ge 1$, $\{\xi_z^N : z \in \bb
Z^d\}$ are i.i.d. random variables with a common $\alpha$-stable
distribution, independent of $N$. Following \cite[Section 3]{fin}, we
may refine this construction to obtain i.i.d.\! random variables
distributed according to any law in the domain of attraction of an
$\alpha$-stable law.

Taking the array $\{\xi_z^N : z \in \bb Z^d\}$, $N\ge 1$, as our
environment, instead of a sequence $\{\xi_z : z \in \bb Z^d\}$ of
i.i.d.\!  random variables in the domain of attraction of an
$\alpha$-stable law, as it is usually done, produces noticeable
differences in the scaling limit, the main one being the survival of
the measure $W$.

\section{Notation and Results}
\label{sec1}

Fix a finite, strictly positive measure $W$ on the $d$-dimensional
torus $\bb T^d$:
\begin{equation}
\label{f15}
W (A) \;>\; 0 \quad\text{for any open set $A$.}
\end{equation}
Denote by $\bb T_N^d$ the $d$-dimensional,
discrete torus $(\bb Z /N \bb Z)^d$. Let $W^N_x$, $x\in\bb T^d_N$, be
the $W$-measure of the $d$-dimensional cube $[x/N, (x+\mb 1)/N)$,
where $\mb 1$ is the vector with all components equal to $1$: $\mb 1 =
(1, \dots, 1)$:
\begin{equation}
\label{f14}
W^N_x \;=\; W\big\{ [x/N, (x+\mb 1)/N)\, \big\}\;.
\end{equation}
We examine in this article the evolution of a continuous time, nearest
neighbor, symmetric random walk on $\bb T^d_N$ which waits a mean
$W^N_x$ exponential time at site $x$.  Its generator $\mc L_N$ is
given by:
\begin{equation}
\label{fc03}
(\mc L_N f)(x) \;=\; \frac 1{2d} \, \frac 1{W^N_x} \, \sum_{y\sim x} [f(y) -
f(x)]\; ,
\end{equation}
for every $f:\bb T^d_N\to\bb R$, where $(y_1, \dots, y_d) = y \sim x =
(x_1, \dots, x_d)$ if $|y-x| = \sum_{1\le i\le d} |x_i-y_i| =1$.

\subsection{Hydrodynamic limit in dimension $1$.}

Consider a finite number of random walks evolving independently on
$\bb T_N$ according to the dynamics defined by the generator $\mc
L_N$. Let $\bb N_0$ be the non-negative integers: $\bb N_0 = \{0, 1,
\dots\}$. Denote by $\Omega_N = \bb N_0^{\bb T_N}$ the state space of
the process and by $\eta$ the configurations of $\Omega_N$ so that
$\eta(x)$, $x\in \bb T_N$, represents the number of particles at site
$x$ for the configuration $\eta$.

This evolution corresponds to a Markov process on $\Omega_N$ whose
generator $L_N$ is given by
\begin{equation*}
(L_N f) (\eta) =  \frac 12 \, \sum_{x\in \bb T_N} \sum_{y \sim x} 
\frac{\eta(x)}{N W^N_x} \big[f(\eta^{x,y})-f(\eta)\big],
\end{equation*}  
where $f:\Omega_N \to \bb R$ is a bounded function and $\eta^{x,y}$
stands for the configuration obtained from $\eta$ by moving a particle
from site $x$ to site $y$:
\begin{equation*}
\eta^{x,y}(z)=
\begin{cases}
\eta(x)-1, &z=x\\
\eta(y)+1, &z=y\\
\eta(z), & z \neq x,y.\\
\end{cases}
\end{equation*}

Notice that we have slowed down the dynamics by a factor $N$. We did
that in order to have a jump rate $N W^N_x$ of order one if the
measure $W$ is absolutely continuous with respect to the Lebesgue
measure in a neighborhood of $x/N$. Indeed, in this case, if we denote
by $w$ the Radon-Nikodym derivative of $W$, $N W^N_x = N \int_{[x/N,
  (x+1)/N)} w(y) dy$ is of order one. In contrast, if $W$ has a point
mass at $x/N$, $N W^N_x$ is of order $N$, which means that particles
wait exponential times of order $N$ at sites where $W$ has point
masses.  Particles are thus trapped on these sites.

Denote by $\{\eta_t : t\ge 0\}$ the Markov process with generator
$L_N$ \emph{speeded up} by $N^2$. Let $D(\bb R_+, \Omega_N)$ be the
space of right continuous trajectories $\xi: \bb R_+\to \Omega_N$ with
left limits, endowed with the Skorohod topology. For a measure $\mu$
on $\Omega_N$, let $\bb P_{\mu}$ be the probability measure on $D(\bb
R_+, \Omega_N)$ induced by the Markov process $\{\eta_t : t\ge 0\}$
starting from $\mu$.

For $\rho \geq 0$, let $\mf P_\rho$ be the Poisson probability
distribution with parameter $\rho$ in $\bb N_0$: $\mf P_\rho\{k\} =
e^{-\rho} \rho^k/k!$, $k\ge 0$. Denote by $\nu^N_\rho$ the product
measure on $\Omega_N$ with marginals defined by
\begin{equation}
\label{f11}
\nu^N_\rho\{\eta : \eta(x) = k\} \;=\; \mf P_{\rho W_x^N}\{k\}
\;,\quad x \in \bb T_N\;, \; k\ge 0\;.
\end{equation}
It is not hard to see that the measures $\nu^N_\rho$ are invariant and
reversible for the generator $L_N$.

Let $\mc M(\bb T)$ be the space of finite positive measures on the
torus $\bb T$, endowed with the weak topology.  Fix $\gamma>0$ and
denote by $\pi^N = \pi^N(\eta) \in \mc M(\bb T)$ the measure obtained
from a configuration $\eta$ by assigning mass $N^{-\gamma}$ to each
particle:
\begin{equation}
\label{f03}
\pi^N \;=\; \frac{1}{N^\gamma} \sum_{x \in \bb T_N} \eta(x)
\, \delta_{x/N}\; , 
\end{equation}
where $\delta_{x/N}$ stands for the Dirac's measure at $x/N$.
For a continuous function $H:\bb T\to \bb R$, denote by $\<\pi^N,
H\>$ the integral of $H$ with respect to $\pi^N$ so that
\begin{equation*}
\<\pi^N, H\> \;=\; \frac{1}{N^\gamma} \sum_{x \in \bb T_N} 
H(x/N) \, \eta(x)\;.
\end{equation*}

Fix a continuous function $u_0:\bb T\to\bb R_+$ and denote by
$\mu^N_{u_0(\cdot)}$ the product measure on $\Omega_N$ with marginals
given by
\begin{equation}
\label{f10}
\mu^N_{u_0(\cdot)}\{\eta : \eta(x) =k\} \; =\; 
\mf P_{u_0(x/N) N^{\gamma} W^N_x} \{k\} 
\;,\quad x \in \bb T_N\;, \; k\ge 0 \; .
\end{equation}
When $u_0$ is constant function equal to $\rho$, we denote
$\mu^N_{u_0(\cdot)}$ simply by $\mu^N_{\rho} = \mu_{\rho}$.  Thus,
under $\mu^N_{u_0(\cdot)}$, $\eta(x)$ has a Poisson distribution with
parameter $u_0(x/N) W^N_x N^{\gamma}$. 

An elementary computation shows that $\<\pi^N, H\>$ converges to $\int
H(x) u_0(x) W(dx)$ in $L^2(\mu^N_{u_0(\cdot)})$ for every continuous
function $H$:
\begin{equation*}
\lim_{N\to\infty} E_{\mu^N_{u_0(\cdot)}} \Big[ \Big( \<\pi^N, H\> -
\int_{\bb T} H(x) u_0(x) W(dx) \Big)^2 \Big]\;=\;0\;.
\end{equation*}

\noindent{\bf The hydrodynamic equation.}
Let $\mc H_{1}$ be the Sobolev space of all functions in $L^2(\bb T)$
with generalized derivative in $L^2(\bb T)$ endowed with the scalar
product $\< \cdot, \cdot\>_{1,2}$ defined by
\begin{equation*}
\< f, g \>_{1,2}  \;=\; \<f,g\> \;+\;
\int_{\bb T} \, (\partial_x f)(x) \, (\partial_x g) (x) \, dx\;,
\end{equation*}
where $\<\cdot, \cdot \>$ stands for the usual scalar product of
$L^2(\bb T)$. It is well known that the space of functions with
continuous partial derivatives of all order is dense in $\mc H_1$. 
Moreover, any function in $\mc H_1$ has a continuous version.

Denote by $L^2(dW)$ the Hilbert space associated to the measure
$W(dx)$, and by $\<f,g\>_W$ the corresponding inner product. 

\renewcommand{\theenumi}{\roman{enumi}}
\renewcommand{\labelenumi}{{\rm (\theenumi)}}

\begin{definition}
\label{s11}
A bounded measurable function $u: [0,T] \times \bb T \to \bb R$ is a
weak solution of 
\begin{equation}
\label{ec2}
\left\{
\begin{array}{l}
{\displaystyle
\frac d{dt} \, u = \frac 12\, \frac d{dW} \frac d{dx} u \;, }\\
{\displaystyle \vphantom{\Big\{}
u(0,\cdot) = u_0(\cdot) \;;}\\
\end{array}
\right.
\end{equation} 
if
\begin{enumerate}
\item It has finite energy: 
\begin{equation*}
\int_0^T \<u_t, u_t\>_{1,2} \, dt \;<\; \infty\;, 
\end{equation*}
\item For any smooth function $G: [0,T]\times \bb T \to \bb R$
  vanishing at $T$, $G_T=0$,
\begin{equation*}
\<G_0,u_0\>_W \;+\; \int_0^T \<\partial_t G_t,u_t\>_W \, dt
\;=\; \frac 12 \, \int_0^T \< \partial_x G_t , \partial_x u_t \>
\, dt \;.
\end{equation*}
\end{enumerate}
\end{definition}

We prove at the end of this article that there is at most one weak
solution of \eqref{ec2}. Denote by $\pi^N_t$, $t\ge 0$, the empirical
measure associated to the state of the process at time $t$:
\begin{equation*}
\pi^N_t \;=\; \frac{1}{N^\gamma} \sum_{x \in \bb T_N} \eta_t(x)
\, \delta_{x/N}\; , 
\end{equation*}
and recall that time has been speeded up by $N^2$.

\begin{theorem}
\label{t2.1}
Let $W$ be a finite, positive measure on $\bb T$ satisfying
\eqref{f15}. Assume that there exists $\gamma_0>0$ such that
\begin{equation}
\tag*{{\bf (H1)}} 
\lim_{N\to\infty} \frac{1}{N^{2+\gamma_0}} \sum_{x \in \bb T_N} 
\frac{1}{W^N_x} \;=\;0\;.
\end{equation}
Fix $\gamma \ge \gamma_0$.  Then, for every $t\ge 0$, every continuous
function $H:\bb T\to\bb R$, and every $\delta>0$,
\begin{equation*}
\lim_{N\to\infty} \bb P_{\mu^N_{u_0(\cdot)}} \Big[ \, \Big| \<\pi^N_t, H\> -
\int_{\bb T} H(x) u(t,x) W(dx) \Big|\, >\delta \, \Big]\;=\;0\;,
\end{equation*}
where $u$ is the unique weak solution of \eqref{ec2}.
\end{theorem}

If the measure $W$ is absolutely continuous with respect to the
Lebesgue measure and its Radon-Nikodym derivative, denoted by $w(x)$,
is strictly positive, $w >0$ a.s., the previous theorem states that the
empirical measure $\pi^N_t$ converges to the measure $\pi(t,dx) =
u(t,x) w(x) dx$, whose density $u$ is solution of
\begin{equation*}
\left\{
\begin{array}{l}
\partial_t u = (1/2) w^{-1} \Delta u \\
u(0,\cdot) = u_0(\cdot) \; .
\end{array}
\right.
\end{equation*}

The proof of the hydrodynamic behavior of the empirical measure
differs sensibly from the usual ones due to the space irregularity of
the environment. The lack of smoothness is reflected in the dynamics
by an erratic time evolution. To overcome this issue, we average not
only in space but also time, investigating the asymptotic behavior of
the measure $\mf M^N$ on $[0,T]\times \bb T$, defined by
\begin{equation*}
\mf M^N \;=\; \int_0^T \frac 1{N^{1+\gamma}} \sum_{x\in \bb T_N}
\frac{\eta_t(x)}{W^N_x} \, \delta_{x/N} \, dt \;,
\end{equation*}
which does not capture space and time discontinuities.

\subsection{Metastable behavior of the trap model in dimension
$d\ge  2$}
\label{ss21}
Fix a finite, strictly positive, atomic measure $W$ on the
$d$-dimensional torus $\bb T^d$:
\begin{equation*}
W \;=\; \sum_{i\ge 1} w_i \, \delta_{x_i}\;,
\end{equation*}
where $\{x_i : i\ge 1\}$ is a dense subset of $\bb T^d$ and
$\sum_{i\ge 1} w_i <\infty$.

Denote by $\{\hat w_i : i\ge 1\}$ the weights of $W$ in decreasing
order so that $\{\hat w_i : i\ge 1\} = \{w_i : i\ge 1\}$ and $\hat w_1
\ge \hat w_2 \ge \cdots$. In case of ties, choose the smallest site
according to some pre-established order. Let $\{\hat x_i : i\ge 1\}$
be the position of the atoms of $W$ corresponding to the weights
$\{\hat w_i : i\ge 1\}$:
\begin{equation*}
W \;=\; \sum_{i\ge 1} w_i \, \delta_{x_i}
\;=\; \sum_{i\ge 1} \hat w_i \, \delta_{\hat x_i}\;.
\end{equation*}

Recall the definition of $W^N_x$ given in \eqref{f14}.  Denote by
$\{X^N_t : t\ge 0\}$ the random walk on $\bb T_N^d$ with generator
$\mc L_N$. Let $D(\bb R_+, \bb T^d_N)$ be the path space of right
continuous trajectories $\omega: \bb R_+\to \bb T^d_N$ with left
limits endowed with the Skorohod topology.  Denote by $\bb P^N_x$,
$x\in \bb T^d_N$, the probability measure on $D(\bb R_+, \bb T^d_N)$
induced by the Markov process $\{X^N_t\}$ starting from $x$.
Expectation with respect to $\bb P^N_x$ is denoted by $\bb E^N_x$.

Denote by $\nu^N$ the unique stationary state of the process $\{X^N_t
: t\ge 0\}$. An elementary computation shows that $\nu^N$ is in fact
reversible and given by
\begin{equation*}
\nu^N(x) \;=\; \frac 1{W(\bb T^d)} W^N_x\;.
\end{equation*}

Enumerate $\bb T^d_N$ according to the weights $\{W^N_x\}$ in
decreasing order:
\begin{equation*}
\bb T^d_N \;=\; \{x^N_1, x^N_2, \dots , x^N_{N^d}\}\;, \quad
W^N_{x^N_1} \;\ge\; W^N_{x^N_2} \;\ge\; \cdots \;\ge\; W^N_{x^N_{N^d}}\;.
\end{equation*}
In case of ties, choose the smallest site according to some
pre-established order. Following \cite{bc1}, we call the sites $x^N_j$,
$j$ fixed, the very deep traps. These are the relevant states of the
trap random walk on the scale observed here.

Since $W(\bb T^d)$ is finite, we may assume that for every $M>0$,
there exists $N_0$ such that $\hat x_j \in [x^N_j/N - (1/2N) \mb 1 ,
x^N_j/N + (1/2N) \mb 1]$, $1\le j\le M$, for all $N\ge N_0$.

To define the trace of the process $\{X^{N}_t : t\ge 0\}$ on a
subset $F$ of $\bb T^d_N$, let $\mc T_N^{F}(t)$, $t\ge 0$, $F\subset\bb
T^d_N$, be the time the process remains in the set $F$ in the interval
$[0,t]$:
\begin{equation*}
\mc T_N^{F}(t) \;:=\;\int_{0}^t \mathbf{1}\{X^N_s\in F\}\, ds\;,
\end{equation*}
and let $\mc S^{F}_N(t)$ be the generalized inverse of $\mc
T^{F}_N(t)$:
\begin{equation*}
\mc S^{F}_N(t) \;:=\;\sup\big\{s\ge 0 : \mc T^{F}_N(s)\le t\big\}\; .
\end{equation*}
It is well known that the process $\{X^{N,F}_t : t\ge 0\}$ defined by
\begin{equation*}
X^{N,F}_t \;=\; X^N (\mc S^{F}_N(t))
\end{equation*}
is a Markov process with state space $F$, called the trace of
$\{X^N_t\}$ on $F$.

Let $\{\bb Y_k : k\ge 0\}$ be the $d$-dimensional, nearest-neighbor,
symmetric, discrete time random walk on $\bb Z^d$ starting from the
origin.  For $d\ge 3$, denote by $v_d$ the probability that $\{\bb
Y_k\}$ never returns to the origin:
\begin{equation*}
v_d \;=\; \bb P_0 \big[ \bb Y_k \not = 0 \text { for all $k\ge 1$ } \big]\;.
\end{equation*}

Let $A^N_M = \{x^N_1, \dots , x^N_{M}\}$, $1 \le M \le N^d$, and
denote by $\{\hat X^{N,M}_t : t\ge 0\}$ the trace of the process
$\{X^{N}_t : t\ge 0\}$ on the set $A^N_M$. The second main result of
this article states that in dimension $d\ge 3$, the trace process
$\{\hat X^{N,M}_t\}$ converges, as $N\uparrow\infty$, to the random
walk on $\{\hat x_1, \dots, \hat x_M\}$ which waits a mean $\hat
w_j/v_d$ exponential time at $\hat x_j$ and then jumps to $\{\hat x_i
: 1\le i\le M\}$ with uniform probability. Note that we do not rule
out the possibility that the process jumps back to the site where it
was.  To state the result, let $\Psi_N : \bb T^d_N\to \bb N$, be
defined by $\Psi_N(x^N_j) =j$ and let $X^{N,M}_t = \Psi_N(\hat
X^{N,M}_t)$.  Clearly, $\{X^{N,M}_t : t\ge 0\}$ is a Markov process on
$\{1, \dots, M\}$.

\begin{theorem}
\label{mt4}
Fix $T>0$ and assume that $d\ge 3$. As $N\uparrow\infty$, the law of
$\{X^{N,M}_t: 0\le t\le T\}$ converges in distribution to a random walk
in $\{1, \dots, M\}$ with generator $\mf L_M$ given by
\begin{equation*}
(\mf L_M f) (i) \;=\; \frac {v_d}{M \hat w_i}
\sum_{j=1}^M [f(j) - f(i)]\; .
\end{equation*}
Moreover,
\begin{equation*}
\lim_{M\to\infty} \limsup_{N\to\infty}
\max_{1\le j\le M} \bb E^N_{x^N_j} \big[ \mc
T^{\Delta_{N,M}}_N(T) \big] \;=\;0\;,
\end{equation*}
where $\Delta_{N,M} = \bb T^d_N \setminus A^N_M$.
\end{theorem}

\medskip In dimension $2$ the picture is similar, but the process
needs to be \emph{speeded up} by $\log N$. Denote by $\{\mf X^N_t :
t\ge 0\}$ the random walk on $\bb T^2_N$ with generator $(\log N) \mc
L_N$, where $\mc L_N$ has been introduced in \eqref{fc03}. Hence,
$\{\mf X^N_t : t\ge 0\}$ has the same distribution as $\{X^N_{t \log
  N} : t\ge 0\}$.

Denote by $\mb P^N_x$, $x\in \bb T^2_N$, the probability measure on
$D(\bb R_+, \bb T^2_N)$ induced by the Markov process $\{\mf X^N_t :
t\ge 0\}$ starting from $x$.  Expectation with respect to $\mb P^N_x$
is denoted by $\mb E^N_x$.  Denote by $\{\hat{\mf X}^{N,M}_t : t\ge
0\}$ the trace of the process $\{\mf X^N_t : t\ge 0\}$ on the set
$A^N_M$ and let $\mf X^{N,M}_t = \Psi_N(\hat{\mf X}^{N,M}_t)$.

\begin{theorem}
\label{mt5}
Fix $T>0$ and assume that $d=2$. As $M\uparrow\infty$, the law of
$\{\mf X^{N,M}_{t}: 0\le t\le T\}$ converges in distribution to a
random walk in $\{1, \dots, M\}$ with generator $\mf L^\star_M$ given
by
\begin{equation*}
(\mf L^\star_M f) (i) \;=\; \frac {\pi}{2} \frac{1}{ M \hat w_i}
\sum_{j=1}^M [f(j) - f(i)]\; .
\end{equation*}
Moreover, if we denote by $\mf T^{\Delta_{N,M}}_N(T)$ the time spent
by the process $\mf X^{N}_{t}$ in the set $\Delta_{N,M}$ on the time
interval $[0,t]$, 
\begin{equation*}
\lim_{M\to\infty} \limsup_{N\to\infty}
\max_{1\le j\le M} \mb E^N_{x^N_j} \big[ \mf
T^{\Delta_{N,M}}_N(T) \big] \;=\;0\;.
\end{equation*}
\end{theorem} 

We prove in Proposition \ref{st1} that in dimension $2$ the random
walk $\{X^N_t : t\ge 0\}$ with generator $\mc L_N$ does not leave a
very deep trap, staying there indefinitely. Therefore, on time scales
of order $1$ the random walk does not move, and on scales of order
$\log N$ the geometry is wiped out and the random walk jumps from a
very deep trap to another one, chosen with uniform probability.

Recall from \cite[Definition 3.1]{fm1} the definition of the
$K$-process, a Markov process on $\overline {\bb N}$, the one-point
compactification of $\bb N$, characterized by two parameters: $c\ge 0$
and a sequence $\{\gamma_i >0: i\ge 1\}$ such that $\sum_{i\ge 1}
\gamma_i <\infty$. While $\gamma_i^{-1}$ represents the rate at which
the Markov process leaves $i$, $c$ is related to the behavior of the
process at the extra point added in the compactification.

Denote by $\{Z^M_t : t\ge 0\}$ the Markov process with generator $\mf
L_M$. Fontes and Mathieu \cite[Lemma 3.11]{fm1} proved that the
process $Z^M_t$ converges, as $M\uparrow\infty$, to the $K$-process
with parameters $c= 0$ and $\{ \hat w_i/v_d: i\ge 1\}$. Next result
follows from this fact and from Theorem \ref{mt4}.

\begin{theorem}
\label{mt2}
Fix $T>0$ and assume that $d\ge 3$. There exists a sequence
$\{\ell^*_N : N\ge 1\}$, $\ell^*_N\uparrow\infty$, such that for any
sequence $\{\ell_N : N\ge 1\}$, $\ell_N \le \ell^*_N$,
$\ell_N\uparrow\infty$, the law of $\{X^{N,\ell_N}_t: 0\le t\le T\}$
converges in distribution to the $K$-process with parameters $\{\hat
w_i/v_d : i\ge 1\}$ and $c=0$. Moreover,
\begin{equation*}
\lim_{N\to\infty} \max_{1\le j\le \ell_N} \bb E^N_{x^N_j} \big[ \mc
T^{\Delta_{N,\ell_N}}_N(T) \big] \;=\;0\;.
\end{equation*}
\end{theorem}

Of course, a similar statement holds in dimension $2$, i.e., for $\mf
L^\star_M$ in place of $\mf L_M$.  In the terminology of Definition
2.1 in \cite{bl1}, Theorem \ref{mt2} states that in dimension $d\ge 3$
the trap random walk $\{X^N_t : t\ge 0\}$ is metastable with
metastates $\{\hat x_1, \hat x_2, \dots \}$ and limit given by the
$K$-process with parameters $\{\hat w_i/v_d : i\ge 1\}$ and
$c=0$. Analogously, in dimension $2$, the trap random walk $\{\mf
X^N_t : t\ge 0\}$ is metastable with metastates $\{\hat x_1, \hat x_2,
\dots \}$ and limit given by the $K$-process with parameters $\{2 \hat
w_i/\pi : i\ge 1\}$ and $c=0$.

\subsection{Bouchaud's trap model}
\label{ss12}

In this subsection we present an example of a random measure $W$ which
satisfies almost surely assumption ({\bf H1}). Fix $0<\alpha <1$ and
let $\lambda$ be the measure on $\bb T^d \times (0,\infty)$ given by
$\lambda = \alpha w^{-(1+\alpha)} dx\, dw$. Since $\lambda$ is a
positive Radon measure, the Poisson point process $\Gamma$ of
intensity $\lambda$ is well defined. Let $\{(x_i,w_i) : i\ge 1\}$ be
the Poisson marks and define the measure $W$ by
\begin{equation*}
W \;=\; \sum_{i\ge 1} w_i \, \delta_{x_i}\;.
\end{equation*}

Note that $W(\bb T^d)$ is $a.s.$ finite. On the one hand, the random
variable $\Gamma(\bb T^d \times (1,\infty))$ has finite mean which
implies that there are only a finite number of Poisson marks on $\bb
T^d\times [1,\infty)$. On the other hand, $\sum_{i\ge 1} w_i \mb
1\{w_i\le 1\}$ has finite expectation. Note also that $W(A)$, $W(B)$
are independent if $A$ and $B$ are disjoints.

Denote by $|A|$ the Lebesgue measure of a measurable set $A \subseteq
\bb T^d$. A simple computation shows that the random variable $W(A)$
has an $\alpha$-stable distribution for any $A$ with $|A|>0$.  In
particular, the random measure $W$ is self-similar with index
$\alpha/d$ in the sense that the distributions of $W(\beta A)$ and
$\beta^{d/\alpha}W(A)$ are the same for any $\beta \in (0,1)$ and any
measurable set $A \subseteq \bb T^d$.

We call the random measure $W(dx)$ a $d$-dimensional subordinator of
index $\alpha$. For $x \in \bb T^d_N$, define
\begin{equation*}
\tau_x^N = N^{d/\alpha}W_x^N \; .
\end{equation*}
Since $W(dx)$ is self-similar with index $\alpha/d$, $\{\tau_x^N;x \in
\bb T^d_N\}$ is a sequence of i.i.d. random variables with common
$\alpha$-stable distribution $\zeta = W(\bb T^d)$ which does not
depend on $N$.

Fontes, Isopi, Newman version of the Bouchaud trap model is the
symmetric, nearest-neighbor, continuous-time random walk on $\bb
T^d_N$ with generator $\mc L_N$ in which $W^N_x$ is replaced by
$\tau_x^N = N^{d/\alpha}W_x^N$:
\begin{equation*}
(\mc L^\tau_N f)(x) \;=\; \frac 1{2d} \, \frac 1{\tau^N_x} \, 
\sum_{y\sim x} [f(y) - f(x)]\; .
\end{equation*}

In dimension $1$, the generator on $\Omega_N$ corresponding to the
superposition of independent random walks is given by
\begin{equation*}
(L^\tau_N f) (\eta) =  \frac 12 \, \sum_{x\in \bb T_N} \sum_{y \sim x} 
\frac{\eta(x)}{\tau^N_x} \big[f(\eta^{x,y})-f(\eta)\big]\;.
\end{equation*}  
Denote by $\{\eta^\tau_t : t\ge 0\}$ the Markov process with generator
$L^\tau_N$ \emph{speeded up} by $N^{1+(1/\alpha)}$ and denote by
$\pi^{N,\tau}_t$ the empirical measure associated to the configuration
$\eta^\tau_t$ by formula \eqref{f03}. Observe that the time scaling is
subdiffusive.

We show below in \eqref{f16} that assumption ({\bf H1}) is in force
almost surely. Moreover, if the Markov process starts from
$\mu^N_{u_0(\cdot)}$, for some continuous function $u_0:\bb T\to \bb
R_+$, by Theorem \ref{t2.1}, for almost all measures $W$, for all
$t\ge 0$, the random measure $\pi^{N,\tau}_t$ converges in probability
to the measure $u(t,x) W(dx)$, where $u$ is the unique weak solution
of \eqref{ec2}. Note that the noise $W$ survives entirely in the
limit, even the differential equation depends on $W$.

%Theo mt2-4
\medskip

We conclude this section showing that assumption ({\bf H1}) is in
force for $\gamma_0 > (1/\alpha) -1$. Indeed, with the notation
introduced above, assumption ({\bf H1}) can be restated as
\begin{equation}
\label{f16}
\frac{1}{N^{2+\gamma-(1/\alpha)}} \sum_{x \in \bb T_N}
\frac{1}{\tau_x^N} \;\longrightarrow\; 0 \quad \text{a.s.}
\end{equation}

It is well known that $1/\tau_x^N$ has finite moments of any order.
Denote by $m_1$ the expectation of $1/\tau_x^N$. The variance of the
previous sum is equal to $N^{-\{3 + 2\gamma - (2/\alpha)\}} \sigma^2$,
for some finite constant $\sigma^2$.  Therefore, by Chebyshev's
inequality, for every $\epsilon >0$,
\begin{equation*}
P\Big[\, \Big|\frac{1}{N^{\{2+\gamma-(1/\alpha)\}}} \sum_{x \in \bb T_N}
\big\{ \frac{1}{\tau_x^N} - m_1 \big\}\, \Big| \;\geq\; \epsilon \Big] 
\;\leq\; \frac{\sigma^2}{\epsilon^2 N^{3 + 2\gamma -
    (2/\alpha)}}\;\cdot 
\end{equation*}
Taking $\epsilon = N^{-\delta}$, for $\delta>0$ small enough, it
follows from Borel-Cantelli that the sum in \eqref{f16} vanishes
a.s.\!  provided $\gamma > (1/\alpha) - 1$.

\section{Proof of the hydrodynamic limit}
\label{s2}

In this section we prove Theorem \ref{t2.1}. Fix $T>0$ and denote by
$\mc M([0,T]\times \bb T)$ the space of finite, positive measures on
$[0,T]\times \bb T$, endowed with the weak topology. For each $N\ge
1$, consider the measure $\mf M^N$ on $[0,T]\times \bb T$ defined by
\begin{equation*}
\mf M^N \;=\; \int_0^T \frac 1{N^{1+\gamma}} \sum_{x\in \bb T_N}
\frac{\eta_t(x)}{W^N_x} \, \delta_{x/N} \, dt \;.
\end{equation*}
Hence, if we denote by $\ll \mf M^N, H\gg$ the integral of a continuous
function $H:[0,T]\times \bb T \to \bb R$ with respect to $\mf M^N$, we
have that
\begin{equation*}
\ll \mf M^N, H\gg \;=\; \int_0^T \frac 1{N^{1+\gamma}} \sum_{x\in\bb T_N}
H(t,x/N) \, \frac{\eta_t(x)}{W^N_x}  \, dt\;.
\end{equation*}

Let $D([0,T], \mc M)$ be the space of right continuous trajectories
$\pi: [0,T]\to \mc M$ with left limits, endowed with the Skorohod
topology.  Fix a continuous function $u_0:\bb T\to \bb R_+$. Let $\mc
Q_N$, $N\ge 1$, be the probability measure on $D([0,T], \mc M) \times
\mc M([0,T]\times \bb T)$ induced by the initial distribution
$\mu^N_{u_0(\cdot)}$ and the pair $(\{\pi^N_t : 0\le t\le T\} , \mf
M^N)$: $\mc Q_N = \bb P_{\mu^N_{u_0(\cdot)}} \circ (\{\pi^N_t : 0\le
t\le T\} , \mf M^N)^{-1}$.  We prove in Lemma \ref{s08} below that the
sequence $\{\mc Q_N : N\ge 1\}$ is tight for the uniform topology in
the first variable, and, in Subsection \ref{ss20}, that all limit
points of the sequence $\{\mc Q_N : N\ge 1\}$ are concentrated on
measures $(\{\pi_t : 0\le t\le T\} , \mf M)$ whose first coordinate is
absolutely continuous with respect to $W$, $\pi(t,dx) = v(t,x) W(dx)$,
and whose density $v_t$ is a weak solution of the hydrodynamic
equation (\ref{ec2}). Since, by Theorem \ref{t2}, there is at most one
weak solution, for each $0\le t\le T$, $\pi^N_t$ converges weakly to
$v(t,x) W(dx)$, where $v_t$ is the unique weak solution of
\eqref{ec2}, as claimed in Theorem \ref{t2.1}.

\subsection{Entropy estimates}
\label{ss21b}

Recall from \cite[Section A1.8]{kl} the definition of the relative
entropy $H(\lambda|\mu)$ of a probability measure $\lambda$ with
respect to another probability measure $\mu$ defined on the same
space, as well as its explicit formula presented in \cite[Theorem
A1.8.3]{kl}.  An elementary computation shows that there exists a
finite constant $K_0$ such that
\begin{equation}
\label{f08}
H(\mu^N_{u_0(\cdot)}|\mu_\rho) \;\le\; K_0 N^\gamma
\end{equation}
for all $N\ge 1$. In fact $N^{-\gamma} H(\mu^N_{u_0(\cdot)}|\mu_\rho)$
converges to $\int \{ u_0(x) \log [ u_0(x)/\rho ] - [u_0(x) - \rho] \}
W(dx)$ as $N\uparrow\infty$.

Denote by $\< \cdot, \cdot \>_{\mu_\rho}$
the scalar product of $L^2(\mu_\rho)$ and denote by $I^W_N$ the
convex and lower semicontinuous \cite[Corollary A1.10.3]{kl}
functional defined by
\begin{equation*}
I^W_N (f) \;=\; \< -L_N \sqrt f \,,\, \sqrt f\>_{\mu_\rho}\; ,
\end{equation*}
for all probability densities $f$ with respect to $\mu_\rho$ (i.e.,
$f\ge 0$ and $\int f d\mu_\rho =1$). An elementary computation shows
that 
\begin{eqnarray*}
\!\!\!\!\!\!\!\!\!\!\!\!\!\! &&
I^W_N (f) \;=\; \sum_{x\in \bb T_N} I^W_{x,x+1} (f)\;,
\quad\text{where}\quad \\
\!\!\!\!\!\!\!\!\!\!\!\!\!\! && \qquad
I^W_{x,x+1} (f) \;=\;  \frac 1{2N} \int \frac {\eta(x)}{W^N_x}
\, \big\{ \sqrt{f(\eta^{x,x+1})} -
\sqrt{f(\eta)} \big\}^2 \, d\mu_\rho  \;.
\end{eqnarray*}
By \cite[Theorem A1.9.2]{kl}, if $\{S^N_t : t\ge 0\}$ stands
for the semi-group associated to the generator $N^2L_N$, for all $t\ge
0$, 
\begin{equation}
\label{f01}
H_N (\mu^N_{u_0(\cdot)} S^N_t | \mu_\rho) \; +\; N^2 \, \int_0^t 
I^W_N (f^N_s) \, ds  \;\le\; H_N (\mu^N_{u_0(\cdot)} | \mu_\rho)\;,
\end{equation}
provided $f^N_s$ stands for the Radon-Nikodym derivative of
$\mu^N_{u_0(\cdot)} S^N_s$ with respect to $\mu_\rho$.

\subsection{Attractiveness and coupling estimates}
\label{ss22}

Let $(\Omega,\mc F, P)$ be a probability space and let $\preceq$ be a
partial order in $\Omega$. We say that a function $f: \Omega \to \bb
R$ is {\em increasing} if $f(\eta) \leq f(\xi)$ whenever $\eta \preceq
\xi$. Let $\lambda$, $\mu$ be two probability measures in $\Omega$. We
say that $\lambda$ is {\em stochastically dominated} by $\mu$ if $\int
f d\lambda \leq \int f d\mu$ for any increasing bounded function $f:
\Omega \to \bb R$. An equivalent definition is the following. We say
that a probability measure $\Lambda$ defined in $\Omega \times \Omega$
is a {\em coupling} of $\lambda$ and $\mu$ if $\Lambda(A \times
\Omega) = \lambda(A)$, $\Lambda(\Omega \times A)=\mu(A)$ for any $A
\in \mc F$. The measure $\lambda$ is stochastically dominated by $\mu$
if there is a coupling $\Lambda$ of $\lambda$ and $\mu$ such that
$\Lambda((\eta,\xi); \eta \preceq \xi)=1$.

We say that a stochastic process $\eta_t$ defined in $\Omega$ is {\em
  attractive} if for any two probability measures $\lambda_1 \preceq
\lambda_2$ there is a process $(\eta_t^1,\eta_t^2)$ in $\Omega \times
\Omega$ such that $\eta_t^i$ is distributed as the process $\eta_t$
with initial distribution $\mu_i$ for $i=1,2$, and such that
$P(\eta_t^1 \preceq \eta_t^2) =1$ for any $t \geq 0$. We call the
process $(\eta_t^1,\eta_t^2)$ a {\em coupling}.

In $\Omega_N$, we say that $\eta \preceq \xi$ if $\eta(x) \leq \xi(x)$
for any $x \in \bb T_N$. For this partial order, it is easy to see
that $\eta_t$ is attractive. Indeed, since the state space is finite,
it is enough to show the existence of a coupling for measures $\mu_i$
concentrated on fixed configurations $\eta^i$, with $\eta^1 \preceq
\eta^2$. Define $\eta_t^1$ as the process $\eta_t$ with initial
configuration $\eta^1$.  Then, define the process $\bar \eta_t$ as a
copy of the process $\eta_t$, independent of $\eta_t^1$, and starting
from $\bar \eta$, where $\bar \eta(x) = \eta^2(x) -\eta^1(x)$. Now
define $\eta_t^2$ by taking $\eta_t^2(x) = \eta_t^1(x) +\bar
\eta_t(x)$.  Since the motion of different particles is independent,
it is clear that $(\eta_t^1,\eta_t^2)$ is the desired coupling as, by
construction, $\eta_t^1 \preceq \eta_t^2$ for any $t \geq 0$.

In terms of stochastic domination, the definition of attractiveness
reads as follows. If $\lambda^1$ is stochastically dominated by
$\lambda^2$, then $\lambda_t^1$ is stochastically dominated by
$\lambda_t^2$ for any time $t \geq 0$, where $\lambda_t^i$ denotes the
distribution in $\Omega_N$ of the process $\eta_t$ with initial
distribution $\lambda^i$. In particular, we obtain the following
inequality, which we call the {\em coupling estimate}:

\begin{proposition}
\label{s04}
Let $\lambda^1$, $\lambda^2$ be two probability measure on $\Omega_N$.
If $\lambda^1$ is stochastically dominated by $\lambda^2$, then
\begin{equation*}
\bb E_{\lambda^1}[F(\eta_t)] \leq \bb E_{\lambda^2} [F(\eta_t)]
\end{equation*}
for any $t \geq 0$ and any bounded increasing function $F: \Omega_N
\to \bb R$.
\end{proposition}

Now we need a criterion to decide whether an initial distribution is
stochastically dominated by another one. In $\bb N_0$, consider the
canonical ordering. It is easy to show that $\mf P_{\rho_1}$ is
stochastically dominated by $\mf P_{\rho_2}$ whenever $\rho_1 \leq
\rho_2$. Since the measures $ \mu_\rho$ are of product form, $
\mu_{\rho_1}$ is stochastically dominated by $ \mu_{\rho_2}$ each time
$\rho_1 \leq \rho_2$. More interesting for us, we have the following.

\begin{proposition}
\label{s03}
Fix an initial bounded non-negative profile $u_0: \bb T \to \bb
R_+$. Define $\bar \rho = ||u_0||_\infty$. Then, $\mu^N_{u_0(\cdot)}$ is
stochastically dominated by $\mu_{\bar \rho}$ for any $N >0$. In
particular,
\begin{equation*}
\bb E_{\mu^N_{u_0(\cdot)}}[F(\eta_t)] \;\leq\; 
\bb E_{\mu_{\bar \rho}}[F(\eta_t)]
\end{equation*}
for any $t\ge 0$ and for any increasing bounded function $F: \Omega_N
\to \bb R$.
\end{proposition}

The coupling shows that $\pi^N_t$, $\mf M^N$ converge to measures
which are absolutely continuous with respect to $W$, the Lebesgue
measures, respectively:

\begin{lemma}
\label{s01}
Every limit point $\mc Q^*$ of the sequence $\mc Q_N$ is concentrated
on measures $\pi(t,dx) = u(t,x)\, W(dx)$ (resp. $\mf M (dt,dx) =
v(t,x) dt dx$) which are absolutely continuous with respect to $W$
(resp. the Lebesgue measure) and whose density $u(t,x)$
(resp. $v(t,x)$) is positive and bounded by $\Vert u_0\Vert_\infty$.
\end{lemma}

\begin{proof}
Fix a limit point $\mc Q^*$ of the sequence $\mc Q_N$ and assume,
without loss of generality, that $\mc Q_N$ converges to $\mc Q^*$ (in
the uniform topology on the first coordinate).  Fix a continuous,
positive function $G: [0,T]\times \bb T \to \bb R$, $\varepsilon >0$
and recall that $\bar\rho = \Vert u_0\Vert_\infty$. By the previous
proposition,
\begin{equation*}
\begin{split}
& \bb P_{\mu^N_{u_0(\cdot)}} \Big[ \ll \mf M, G\gg \;\ge\; 
\bar\rho \int_0^Tdt  \int_{\bb T} G(t,x) dx  + \varepsilon \Big] \\
& \qquad \;\le\; 
\bb P_{\mu_{\bar \rho}} \Big[ \ll \mf M, G \gg  \;\ge\; \bar\rho 
\int_0^Tdt \int_{\bb T} G(t,x) dx +  \varepsilon \Big]
\end{split}
\end{equation*}
for every $N\ge 1$. We may replace the integral $\int_{\bb T} G(t, x)
du$ by the Riemann sum because $G$ is continuous. Thus, for $N$ large
enough, the previous expression is bounded above by
\begin{equation*}
\bb P_{\mu_{\bar \rho}} \Big[  \int_0^T \frac 1{N^{1+\gamma}} 
\sum_{x\in \bb T_N} G(t,x/N) \,
\Big\{ \frac{\eta_t(x)}{W^N_x} - \bar\rho \, N^{\gamma} \Big\}\, dt
\;\ge\; \varepsilon/2 \Big]\;.
\end{equation*}
By Chebyshev and by Schwarz inequalities, since $\mu_{\bar \rho}$ is a
stationary state given by a product of Poisson measures, this
expression is less than or equal to
\begin{equation*}
\frac {4T}{\varepsilon^2 } \int_0^T \frac {\bar\rho}{N^{2 + \gamma}} 
\sum_{x\in \bb T_N} G(t,x/N)^2 \frac 1{W^N_x}\, dt \; .
\end{equation*}
In view of assumption ({\bf H1}), this expression vanishes as
$N\uparrow\infty$ because $G$ is a continuous bounded function. 

Since $\mc Q_N$ converges to $\mc Q^*$, for every $\varepsilon>0$, 
\begin{equation*}
\mc Q^* \Big[ \ll \mf M, G \gg \;\ge\; 
\bar\rho \int_0^Tdt \int_{\bb T} G(t,x)\,
dx  + \varepsilon \Big] \;=\; 0\;.
\end{equation*}
Letting $\varepsilon \downarrow 0$, we conclude that $\mc Q^*$ is
concentrated on measures $\mf M$ such that $\ll \mf M, G \gg \le
\bar\rho \int_0^Tdt \int_{\bb T} G(t,x) \, dx$. Taking a set $\{G_k :
k\ge 1\}$ of positive, bounded, continuous functions dense for the
uniform topology, we conclude that $\mc Q^*$ is concentrated on
absolutely continuous measures $\mf M(dt, dx) = v(t,x) dt \, dx$,
whose density $v(t,x)$ is bounded by $\bar\rho$.

A similar coupling argument shows that for every $0\le t\le T$ and
every  continuous, positive function $H:\bb T\to \bb R$,
\begin{equation*}
\lim_{N\to\infty}
\bb P_{\mu^N_{u_0(\cdot)}} \Big[ \< \pi^N_t , H\> \;\ge\; 
\bar\rho \int_{\bb T} H(x) \, W(dx)  + \varepsilon \Big] \;=\;0\;.
\end{equation*}
Since we assumed compactness in the uniform topology, we deduce from
this formula that
\begin{equation*}
\mc Q^* \Big[ \< \pi_t , H\> \;\ge\; 
\bar\rho \int_{\bb T} H(x) \, W(dx)  + \varepsilon \Big] \;=\;0\;.
\end{equation*}
It remains to recall the arguments presented for $\mf M$ to conclude
the proof.
\end{proof}

\subsection{Hydrodynamic limit}
\label{ss20}

We prove in this subsection Theorem \ref{t2.1}.

\begin{theorem}
\label{s09}
The sequence of probability measures $\{\mc Q_N : N\ge 1\}$ converges
to the measure $\mc Q^*$ concentrated on the absolutely continuous
pair $(\{\pi_t :0\le t\le T\}, \mf M)$, $\pi_t = v(t,x) W(dx)$, $\mf M
= v(t,x) \, dt\, dx$, whose density $v(t,x)$ is the weak solution of
the equation \eqref{ec2}.
\end{theorem}

\begin{proof}
By Lemma \ref{s08} below, the sequence $\{\mc Q_N : N\ge 1\}$ is
tight. Fix a limit point $\mc Q^*$ and assume, without loss of
generality, that $\mc Q_N$ converges to $\mc Q^*$.  

Fix a smooth function $H: [0,T] \times \bb T \to \bb R$ such that
$H(T, \cdot) = 0$. Consider the martingale $M^H_N(t)$ defined by
\begin{equation}
\label{f02}
M^H_N(t) \;=\; \<\pi^N_t , H_t\>  \;-\; \<\pi^N_0 , H_0\> \;-\; 
\int_0^t \<\pi^N_s , \partial_s H_s \> \, ds
\;-\; N^2 \int_0^t L_N \<\pi^N_s , H\> \, ds \; .
\end{equation}
The variance of this martingale is equal to
\begin{equation*}
\frac{N}{2 N^{2\gamma}} \sum_{x\in \bb T_N}\sum_{y: |y-x|=1} \int_0^t
\frac{\eta_s(x)}{W^N_x} \{ H(s,y/N) - H(s,x/N)\}^2\, ds \;.
\end{equation*}
The coupling estimate shows that the expectation of this expression
with respect to $\bb P_{\mu^N_{u_0(\cdot)}}$ is bounded by $C_0
N^{-\gamma}$ for some finite constant $C_0$ which depends on $H$ and
$\bar\rho$. On the other hand, an elementary computation shows that
\begin{equation*}
N^2 \int_0^T L_N \<\pi^N_s , H\> \, ds \;=\;
\frac 12 \ll \mf M^N , \Delta_N H\gg\;,
\end{equation*}
where $\Delta_N$ stands for the discrete Laplacian. In particular, in
view of \eqref{f02} and since $H(T, \cdot)$ vanishes, for every
$\delta>0$,
\begin{equation*}
\lim_{N\to\infty} \bb P_{\mu^N_{u_0(\cdot)}} \Big[
\, \Big| \<\pi^N_0 , H_0\>  \;+\; \int_0^T \<\pi^N_s , \partial_s H_s\> \, ds
\;+\; (1/2) \ll \mf M^N , \Delta_N H \gg \, \Big| >\delta \Big]
\;=\; 0\;.
\end{equation*}

The first term of this sum converges to $\int_{\bb T} H_0(x) u_0(x)
W(dx)$ in $\bb P_{\mu^N_{u_0(\cdot)}}$-probabil\-ity, as
$N\uparrow\infty$. The last expression can be written, up to smaller
order terms, as $(1/2) \ll \mf M^N , \Delta H\gg$.  Hence, since $\mc
Q^N$ converges to $\mc Q^*$, for every $\delta>0$, and every smooth
function $H$,
\begin{equation*}
\mc Q^* \Big[
\, \Big| \int_{\bb T} H_0(x) u_0(x) W(dx)  
\;+\; \int_0^T \<\pi_s , \partial_s H_s\> \, ds
\;+\; (1/2) \ll \mf M , \Delta H \gg \, \Big| >\delta \Big]
\;=\; 0\;.
\end{equation*}
Letting $\delta\downarrow 0$, by Lemma \ref{s01}, $\mc Q^*$ almost surely, 
\begin{equation*}
\begin{split}
& \int_{\bb T} H_0(x) u_0(x) W(dx)  
\;+\; \int_0^T  ds \, \int_{\bb T}(\partial_s H)(s,x)\, u(s,x) \, W(dx) \\
&\qquad \;+\; (1/2) \int_0^T ds \, \int_{\bb T} (\Delta H)  
(s,x)\, v(s,x) \, dx \;=\; 0\;.
\end{split}
\end{equation*}
According to Lemma \ref{s06}, we may replace $u$ by $v$ in the second
term. By Proposition \ref{s07}, we may integrate by parts the last
term to obtain that
\begin{equation*}
\< H_0 , u_0 \>_W  
\;+\; \int_0^T \<\partial_s H_s , u_s \>_W \, ds 
\;-\; (1/2) \int_0^T \<\partial_x H_s , \partial_x v_s \> \, ds 
\;=\; 0\;.
\end{equation*}
This proves that $\mc Q^*$ is concentrated on weak solutions of
\eqref{ec2}. By Proposition \ref{s07}, $\partial_x v$ belongs to
$L^2([0,T]\times \bb T)$ and by Lemma \ref{s01} $v$ is positive and
bounded. Since the previous identity holds for all smooth functions
$H$, $v$ is a weak solution of \eqref{ec2}.
\end{proof}

Theorem \ref{t2.1} follows from this result and the tightness in the
uniform topology of the sequence $\{\mc Q_N : N\ge 1\}$ proved in
Lemma \ref{s08} below.

\begin{lemma}
\label{s08}
The sequence $\{\mc Q_N : N\ge 1\}$ is tight in the uniform topology
in the first coordinate.
\end{lemma}

\begin{proof}
To prove tightness of the sequence $\{\mc Q_N : N\ge 1\}$ we need to
examine the two coordinates separately.

Clearly, the sequence of random measures $\mf M^N$ is tight if and
only if the sequence of random variables $\ll \mf M^N, G\gg$ is tight
for every continuous function $G:[0,T]\times \bb T\to\bb
R$. Tightness of the sequence $\ll \mf M^N, G\gg$ follows from a
coupling argument similar to the one used in the proof of Lemma
\ref{s01}. 

To prove tightness of the sequence of processes $\{\pi^N_t : 0\le t\le
T\}$ in the uniform topology, it is enough to examine the process
$\<\pi^N_t, H\>$ for some fixed smooth function $H$. Recall the
definition of the martingale $M^H_N(t)$ introduced in
\eqref{f02}. Tightness of $\<\pi^N_t, H\>$ follows from tightness of
the martingale $M^H_N(t)$ and tightness of the additive functional
$\int_0^t N^2 L_N \<\pi^N_s , H\> \, ds$.

The martingale is tight in the uniform topology because, by Doob
inequality and by the explicit computation of the quadratic variation
of $M^H_N(t)$, for every $\delta>0$
\begin{equation*}
\lim_{N\to\infty} \bb P_{\mu^N_{u_0(\cdot)}} \Big[
\sup_{0\le t\le T} \big| M^H_N(t)  \big| >\delta \Big]
\;=\; 0\;.
\end{equation*}
On the other hand, computing $N^2 L_N \<\pi^N_r , H\> $, by Chebyshev
and Schwarz inequalities, for every $\delta>0$,
\begin{equation*}
\begin{split}
& \bb P_{\mu^N_{u_0(\cdot)}} \Big[ \sup_{0\le |t-s| \le \epsilon} 
\Big| \int_s^t N^2 L_N \<\pi^N_r , H\>
\, dr  \Big| >\delta \Big] \\
&\qquad \;\le\; \frac {\epsilon\, C_0} {\delta^2}
\, \bb E_{\mu^N_{u_0(\cdot)}} \Big[  
\int_0^T \Big\{ \frac 1{N^{1+\gamma}} \sum_{x\in\bb T_N} \frac
{\eta_s(x)}{W^N_x} \Big\}^2  \, ds \Big]
\end{split}
\end{equation*}
for some finite constant $C_0$ which depends only on $H$. By the
coupling estimate, we may replace the measure $\mu^N_{u_0(\cdot)}$ by
the stationary measure $\mu^N_{\bar \rho}$, estimating the expectation
by
\begin{equation*}
C(\bar\rho) T  \Big \{ 1 \;+\; \frac 1{N^{2+\gamma}} \sum_{x\in\bb
  T_N} \frac 1{W^N_x} \Big\}\;.
\end{equation*}
By assumption ({\bf H1}), this expression is bounded uniformly in $N$,
which concludes the proof of tightness.
\end{proof}

\section{Entropy estimates}
\label{sec3}

We prove in this section the main estimates needed in the proof of
hydrodynamic limit. For $\ell \ge 1$, let $\Lambda_\ell$ be a cube of
length $\ell$: $ \Lambda_\ell= \{1 , \dots, \ell\}$ and let
$\Lambda_{x,\ell} = x + \Lambda_\ell$. Denote by $M^\ell (x)$ the
number of particles on $\Lambda_{x,\ell}$ and by $W_N(x,\ell)$ the
$W$-measure of the cube $\Lambda_{x,\ell}$ rescaled by $N$:
\begin{equation*}
M^\ell (x) \;=\; \sum_{y\in \Lambda_\ell} \eta(x+y)\; , \quad
W_N(x,\ell) \; =\; \sum_{y\in \Lambda_\ell} W^N_{x+y}\; .
\end{equation*}
Note that $W_N(x,\epsilon N) \sim W_\epsilon(x)$.

\subsection{Two blocks estimate}
\label{ss23}

We prove in this subsection the so called two blocks estimate. 

\begin{lemma}
\label{s02}
Fix a bounded function $G: [0,T]\times \bb T \to \bb R$.
\begin{equation*}
\lim_{\epsilon \to 0} \limsup_{N\to\infty}
\bb E_{\mu^N_{u_0(\cdot)}} \Big[\,
\Big| \int_0^T  \frac 1{N^{1+\gamma}} \sum_{x\in \bb T_N} G(s,x/N) 
\Big\{ \frac{\eta_s(x)}{W^N_x}  \;-\;
\frac{M^{\epsilon N}_s(x)} {W_N(x,\epsilon N)}
\Big\} \, ds \Big| \, \Big] \;=\;0\;.
\end{equation*}
\end{lemma}

\begin{proof}
Fix a bounded function $G: [0,T]\times \bb T \to \bb R$, $\delta >0$
and a positive constant $C_1 = C_1(\delta)$ to be specified later.
Let
\begin{eqnarray*}
\!\!\!\!\!\!\!\!\!\!\!\!\! &&
V^0_\epsilon (s,\eta) \;=\;  
\frac {1}{N^{1+\gamma}} \sum_{x\in \bb T_N} G(s,x/N) 
\Big\{ \frac{\eta (x)}{W^N_x}  \;-\;
\frac{M^{\epsilon N} (x)} {W_N(x,\epsilon N)} \Big\} \; , \\
\!\!\!\!\!\!\!\!\!\!\!\!\! && \quad
R_\epsilon (s,\eta) \;=\; \frac {C_1 \epsilon}{N^{2+\gamma}} 
\sum_{x\in \bb T_N} G(s,x/N)^2 
\sum_{y=0}^{\epsilon N} \frac{\eta(x+y)}{W^N_{x+y}} \; \cdot
\end{eqnarray*}

By the coupling estimate,  
\begin{equation}
\label{f13}
\bb E_{\mu^N_{u_0(\cdot)}} \Big[\, \int_0^T R_\epsilon (s,\eta_s) ds\,
\Big] \; \le\; C_0 \epsilon^2 \, T
\end{equation}
for some finite constant $C_0$ depending only on $C_1$, $G$,
$\rho$. It is therefore enough to prove that
\begin{equation*}
\lim_{\epsilon \to 0} \limsup_{N\to\infty} 
\bb E_{\mu^N_{u_0(\cdot)}} \Big[\,
\Big| \int_0^T V^0_\epsilon(s,\eta_s) \, ds \Big| 
\;-\; \int_0^T R_\epsilon(s,\eta_s) \, ds \, \Big] \;=\;0\;.
\end{equation*}

By the entropy inequality, Jensen inequality and the entropy estimate
\eqref{f08}, the previous expectation is bounded above by 
\begin{equation*}
\frac {K_0}A \;+\; \frac 1{AN^\gamma}
\log \, \bb E_{\mu_{\rho}} \Big[ \exp AN^\gamma  \Big\{ \Big| 
\int_0^T V_\epsilon^0 (s,\eta_s) \, ds \Big| 
\;-\; \int_0^T R_\epsilon(s,\eta_s) \, ds \Big\} \, \Big]
\end{equation*}
for every positive $A>0$.

Let $A = K_0 \delta^{-1}$. Since $e^{|x|} \le e^x + e^{-x}$ and since
$\limsup_{n \to \infty} N^{-\gamma} \log \big\{ a^1_N + a^2_N \big\}$
$= \max_{i=1,2}$ $\limsup_{n \to \infty} N^{-\gamma} \log a_N^i$, to
prove the lemma it is enough to show that for every $\delta>0$,
\begin{equation*}
\lim_{\epsilon \to 0} \limsup_{n \to \infty} \frac{1}{A N^\gamma} 
\log \, \bb E_{\mu_{\rho}} \Big[ \exp\Big\{ 
AN^\gamma \int_0^T V_\epsilon(s,\eta_s) \, ds \Big\} \, \Big] \;\leq\; 0\;,
\end{equation*}
where $V_\epsilon = V_\epsilon^0 - R_\epsilon$.

By classical arguments, relying on Feynman-Kac's formula
(cf. \cite[p. 267]{kl}), the previous expectation is bounded above by
\begin{equation*}
\int_0^T \sup_{f} \Big\{  \int V_\epsilon(s,\eta) f(\eta)
\mu_{\rho} (d\eta)  - \frac{N^2}{A N^\gamma} I^W_N(f) \Big\}\, ds \;,
\end{equation*}
where the supremum is carried over all densities $f$ with respect to
$\mu_{\rho}$. Hence, to conclude the proof of the lemma, it is enough
to show that
\begin{equation}
\label{f07}
\int V^0_\epsilon(s,\eta) f(\eta) \mu_{\rho} (d\eta) 
\;\le\; \int R_\epsilon(s,\eta) f(\eta) \mu_{\rho} (d\eta)
\; +\; \frac{\delta N^2}{K_0 N^\gamma} I^W_N(f)
\end{equation}
for every density function $f$ and every $\delta>0$.

Recall the definition of $W_N(x,\epsilon N)$ to rewrite
$V^0_\epsilon(s,\eta)$ as 
\begin{eqnarray*}
\frac 1{N^{1+\gamma}} \sum_{x\in \bb T_N} G(s,x/N) 
\sum_{y=1}^ {\epsilon N}  \frac{W^N_{x+y}}{W_N(x,\epsilon N)} 
\Big\{ \frac{\eta(x)}{W^N_x}
-  \frac{\eta(x+y)}{W^N_{x+y}} \Big\} \; . 
\end{eqnarray*}
Fix a density $f$ with respect to $\mu_\rho$.  Performing a simple
change of variables, we see that
\begin{eqnarray*}
\int \Big\{ \frac{\eta(x)}{W_x^N} - \frac{\eta(x+y)}{W_{x+y}^N} \Big\} \,
f\, d\mu_\rho &=& \sum_{z=0}^{y-1}\int \Big\{ 
\frac{\eta(x+z)}{W_{x+z}^N} - \frac{\eta(x+z+1)}{W_{x+z+1}^N}  \Big\} \,
f\, d\mu_\rho \\
&=& \sum_{z=0}^{y-1}\int \frac{\eta(x+z)}{W_{x+z}^N} 
\big\{ f(\eta) - f(\sigma^{x+z, x+z+1} \eta)\big\} \, d\mu_\rho\;.
\end{eqnarray*}
Since $(a-b) = (\sqrt a - \sqrt b)(\sqrt a + \sqrt b)$, by Schwarz
inequality, the previous expression is less than or equal to
\begin{eqnarray*}
\frac {N}{\beta} \sum_{z=0}^{y-1} I^W_{x+z, x+z+1} (f) \;+\;
\frac \beta 2 \sum_{z=0}^{y-1} \int \frac{\eta(x+z)}{W_{x+z}^N} 
\Big\{ \sqrt{f(\sigma^{x+z,x+z+1} \eta)} + \sqrt {f(\eta)} \Big\}^2
\, d\mu_\rho
\end{eqnarray*}
for all $\beta>0$. The same change of variables permit to estimate the
second term as
\begin{eqnarray*}
\beta \, \sum_{z=0}^{y-1} \int \Big\{ \frac{\eta(x+z)}{W_{x+z}^N}
+ \frac{\eta(x+z+1)}{W_{x+z+1}^N} \Big\} \, f(\eta)
\, d\mu_\rho\;.
\end{eqnarray*}

It follows from the previous estimates that for any density $f$ with
respect to $\mu_{\rho}$, and all $\beta>0$,
\begin{eqnarray}
\label{f05}
\!\!\!\!\!\!\!\!\!\!\!\!\!\!\!\! &&
\int V^0_\epsilon(s,\eta) f(\eta) \mu_{\rho} (d\eta) \;\le\;
\frac 1{\beta N^{\gamma}} \sum_{x\in \bb T_N} 
\sum_{y=1}^{\epsilon N} \frac{W^N_{x+y}}{W_N(x,\epsilon N)} 
\sum_{z=0}^{y-1} I^W_{x+z, x+z+1}(f) \\
\!\!\!\!\!\!\!\!\!\!\!\!\!\!\!\! &&
\quad +\; \frac {2 \beta}{N^{1+\gamma}} \sum_{x\in \bb T_N} 
G(s,x/N)^2  \sum_{y=1}^{\epsilon N} \frac{W^N_{x+y}}{W_N(x,\epsilon N)} 
\sum_{z=0}^{y} \int \frac{\eta(x+z)}{W_{x+z}^N} \, f(\eta)
\, d\mu_\rho \;. 
\nonumber
\end{eqnarray}

We examine each term on the right hand side separately. Set $\beta=
2 \epsilon N^{-1} A$.  Changing the order of summation, we obtain that
the second term is less than or equal to
\begin{eqnarray*}
\frac {4 \epsilon A}{N^{2+ \gamma}} \sum_{x\in \bb T_N} 
G(s,x/N)^2  \int \sum_{y=0}^{\epsilon N} \frac{\eta(x+y)}{W_{x+y}^N}
\, f(\eta) \, d\mu_\rho\; .
\end{eqnarray*}
This expression is bounded by the expectation of $R_\epsilon$ with
respect to $f(\eta) \, d\mu_\rho$ provide we choose $C_1 \ge 4A = 4
K_0 \delta^{-1}$. By similar reasons, the first term on the right hand
side of \eqref{f05} is bounded above by
\begin{eqnarray*}
\frac {N(\epsilon N +1)} {2 A \epsilon N^{\gamma}} \sum_{z\in \bb T_N} 
I^W_{z, z + 1}(f)\;.
\end{eqnarray*}
Hence, \eqref{f07} holds and the lemma is proved.
\end{proof}

Consider a sequence $\{G_{N, \epsilon} : N\ge 1, \epsilon >0\}$ of
functions $G_{N, \epsilon}: [0,T] \times \bb T_N \to \bb R$. In the
proof of the two blocks estimate, the boundedness assumption on $G$
was used only at \eqref{f13}.  In particular, the proof presented
above shows that
\begin{equation}
\label{f12}
\begin{split}
& \lim_{\epsilon \to 0} \limsup_{N\to\infty} \\
& \quad \bb E_{\mu^N_{u_0(\cdot)}} \Big[\,
\Big| \int_0^T 
\frac 1{N^{1+\gamma}} \sum_{x\in \bb T_N} 
G_{N, \epsilon} (s, x) \Big\{ 
\frac{M^{\epsilon N}_s(x)} {W_N(x,\epsilon N)} \;-\;
\frac{\eta_s(x)}{W^N_x} \Big\} \, ds \Big| \, \Big] \;=\;0\;.
\end{split}
\end{equation}
provided
\begin{equation*}
\lim_{\epsilon \to 0} \limsup_{N\to\infty} \int_0^T 
\frac {\epsilon^2} {N} \sum_{x\in \bb T_N} 
G_{N, \epsilon} (s, x)^2  \, ds \;=\;0\;.
\end{equation*}

Recall that $W_\epsilon :\bb T\to\bb R$ is defined by $W_\epsilon
(x) = W([x,x+ \epsilon])$.

\begin{corollary}
\label{s10}
Let $J: [0,T] \times \bb T \to \bb R$ be a continuous function. Then,
\begin{equation*}
\lim_{\epsilon \to 0} \limsup_{N\to\infty}
\bb E_{\mu^N_{u_0(\cdot)}} \Big[\, \Big| \int_0^T  \<\pi^N_s , J_s\> \, ds 
\;-\; \ll \mf M^N , J \, \epsilon^{-1} \, W_\epsilon \gg  \Big| 
\, \Big] \;=\;0\;.
\end{equation*}
\end{corollary}

\begin{proof}
Since $J$ is a continuous function,
\begin{equation*}
\<\pi^N_s , J_s\> \;-\; \frac 1{N^\gamma} \sum_{x\in \bb
  T_N} \eta_s(x) \, \frac 1{(\epsilon N)} 
\sum_{-y \in \Lambda_{-x,\epsilon N}} J(s,y/N) 
\end{equation*}
is absolutely bounded by $C(\epsilon) N^{-\gamma} \sum_{x\in \bb
  T_N} \eta(x)$ for some finite constant $C(\epsilon)$ which
vanishes as $\epsilon\downarrow 0$. In particular, by the usual
coupling estimate and changing the order of summation, we get that
\begin{equation*}
\begin{split}
& \lim_{\epsilon \to 0} \, \sup_{N\ge 1}  \\
& \quad \bb E_{\mu^N_{u_0(\cdot)}} \Big[
\, \Big| \int_0^T \<\pi^N_s , J_s\> \, ds
\;-\; \int_0^T  \frac {\epsilon^{-1}}{N^{1+\gamma}}  
\sum_{x\in \bb T_N} J(s, x/N) M^{\epsilon N}_s(x)\, ds\, \Big| 
\, \Big] \;=\; 0\;.
\end{split}
\end{equation*}
 
The second term inside the absolute value can be rewritten as
\begin{equation*}
\frac 1{N^{1+\gamma}} \sum_{x\in \bb T_N} J(s, x/N) 
\epsilon^{-1} W_N(x,\epsilon N) \frac{M^{\epsilon N}_s(x)}
{W_N(x,\epsilon N)} \;\cdot
\end{equation*}
Let $G_{N,\epsilon} (s,x/N) = J(s, x/N) \epsilon^{-1} W_N(x,\epsilon
N)$. Since $J$ is a bounded function, by definition of $W_N(x,\epsilon
N)$,
\begin{equation*}
\int_0^T \frac {\epsilon^2} {N} \sum_{x\in \bb T_N} 
G_{N, \epsilon} (s, x)^2  \, ds \;\le\; 
\frac {C_0 T} {N} \sum_{x\in \bb T_N} W_{2\epsilon} (x/N)^2
\end{equation*}
for some finite constant $C_0$ which depends only on $J$. As
$N\uparrow\infty$ this expression converges to 
\begin{equation*}
C_0 T \int_{\bb T} W_{2\epsilon} (x)^2\, dx\;.
\end{equation*}
For Lebesgue almost all $x$, $W_{2\epsilon} (x)$ vanishes as $\epsilon
\downarrow 0$. Therefore, by \eqref{f12},
\begin{equation*}
\begin{split}
& \lim_{\epsilon \to 0} \limsup_{N\to\infty} \\
& \bb E_{\mu^N_{u_0(\cdot)}} \Big[\,
\Big| \int_0^T 
\frac {\epsilon^{-1}}{N^{1+\gamma}} \sum_{x\in \bb T_N} 
J(s, x/N)  W_N(x,\epsilon N) \Big\{ 
\frac{M^{\epsilon N}_s(x)} {W_N(x,\epsilon N)} \;-\;
\frac{\eta_s(x)}{W^N_x} \Big\} \, ds \Big| \, \Big] \;=\;0\;.
\end{split}
\end{equation*}
Finally, since 
\begin{equation*}
\frac 1{N^{d}} \sum_{x\in \bb T_N} \Big\vert W_N(x,\epsilon N)
- W_{\epsilon} (x/N)\Big\vert
\end{equation*}
vanishes as $N\uparrow\infty$, we may replace $W_N(x,\epsilon N)$ by
$W_{\epsilon} (x/N)$ to conclude the proof of the corollary.
\end{proof}

\subsection{Energy estimate}
\label{ss27}

We proved in Lemma \ref{s01} that any limit point $\mc Q^*$ of the
sequence $\{\mc Q_N : N\ge 1\}$ is concentrated on measures $\mf M$
which are absolutely continuous with respect to the Lebesgue measures:
$\mf M= v(t,x) dt\, dx$. We show in this section that the density $v$
has a generalized space derivative in $L^2([0,T]\times \bb T)$.

\begin{proposition}
\label{s07} 
Any limit point $\mc Q^*$ of the sequence $\{\mc Q_N : N\ge 1\}$ is
concentrated on measures $\mf M = v(t,x) dt\, dx$ with the property
that there exists a function $F$ in $L^2([0,T]\times \bb T)$ such that
\begin{equation*}
\int_0^T ds\, \int_{\bb T} (\partial_{x} H)(s,x)
v(s,x) \, dx \; =\; - \int_0^T ds\, \int_{\bb T} 
H(s,x) F(s,x) \, dx
\end{equation*}
for all smooth functions $H$. We denote the generalized derivative
$F$ of $v$ by $\partial_{x} v$.
\end{proposition}

The proof of this proposition relies on the following estimate.

\begin{lemma}
\label{t4} 
Fix a set of smooth functions $H_i : [0,T]\times \bb T \to \bb R$,
$1\le i\le \ell$. Let
\begin{equation*}
V_i (s, \eta) \;=\;  \frac{1}{N^{\gamma}} \sum_{x \in \bb T_N} 
H_i(s, x/N) \Big\{ \frac{\eta (x)}{W_x^N}
- \frac{\eta (x+1)}{W_{x+1}^N} \Big\}\;.
\end{equation*}
Then, for any $\beta>0$,
\begin{equation*}
\limsup_{N \to \infty} \bb E_{\mu^N_{u_0(\cdot)}} \Big[ \max_{1\le i\le l}
\int_0^T   \Big\{  V_i(s, \eta) \; -\; 
\frac{2 \beta }{N^{1+\gamma}} \sum_{x \in \bb T_N} H_i(s,x/N)^2 
\frac{\eta_s (x)}{W_x^N} \Big\} \, ds \Big] 
\;\leq\; \frac{K_0}{\beta} \;,
\end{equation*}
where $K_0$ is the constant given by \eqref{f08}.
\end{lemma}

\begin{proof}
Fix $\beta>0$ and let
\begin{equation*}
X_i(s, \eta) \;=\; V_i(s, \eta) \; -\; 
\frac{\beta }{N^{1+\gamma}} \sum_{x \in \bb T_N} H_i(s,x/N)^2 
\Big\{ \frac{\eta (x)}{W_x^N} + \frac{\eta (x+1)}{W_{x+1}^N}
\Big\} \;\cdot
\end{equation*}
A summation by parts and a coupling estimate similar to the one used
in the proof of Lemma \ref{s01} shows that it is enough to prove that
\begin{eqnarray*}
\limsup_{N \to \infty} \bb E_{\mu^N_{u_0(\cdot)}} \Big[ \max_{1\le i\le l}
\int_0^T   X_i(s, \eta) \, ds \Big] 
\;\leq\; \frac{K_0}{\beta} \;.
\end{eqnarray*}

By the entropy inequality, Jensen inequality and the entropy estimate
\eqref{f08},
\begin{equation*}
\begin{split}
& \bb E_{\mu^N_{u_0(\cdot)}} \Big[\max_{1\le i\le \ell}
\int_0^T X_i (s,\eta_s) \, ds \Big] \\
& \qquad \le\; \frac{K_0}{\beta} \;+\; \frac 1{\beta N^\gamma}
\log \, \bb E_{\mu_{\rho}} \Big[ \exp\Big\{ \max_{1\le i\le \ell}
\beta N^\gamma \int_0^T  X_i (s,\eta_s) \, ds \Big\} \, \Big]
\end{split}
\end{equation*}
for every $\beta>0$.

Since, on the one hand, $\exp\{\max_{1\le i\le \ell} a^i_N\} \le
\sum_{1\le i\le \ell} \exp\{ a^i_N \}$ and, on the other hand,
$\limsup_{n \to \infty} N^{-\gamma} \log $ $\big\{ \sum_{1\le i\le
  \ell} b^i_N \big\} = \max_{1\le i\le \ell} \limsup_{n \to \infty}
N^{-\gamma} \log b_N^i$, to prove the lemma it is enough to show that
\begin{equation}
\label{f09}
\limsup_{n \to \infty} \frac{1}{\beta N^\gamma} 
\log \, \bb E_{\mu_{\rho}} \Big[ \exp\Big\{ 
\beta N^\gamma \int_0^T X_i (s,\eta_s) \, ds \Big\} \, \Big] \;\leq\; 0
\end{equation}
for $1\le i\le \ell$ and any $\beta >0$. 

By classical arguments, relying on Feynman-Kac's formula
(cf. \cite[p. 267]{kl}), the previous expectation is bounded above by
\begin{equation*}
\int_0^T \sup_{f} \Big\{  \int X_i (s,\eta) f(\eta)
\mu_{\rho} (d\eta)  - \frac{N^2}{\beta N^\gamma} I^W_N(f) \Big\}\, ds \;,
\end{equation*}
where the supremum is carried over all densities $f$ with respect to
$\mu_{\rho}$.  

Therefore, to conclude the proof of the lemma, it is enough to show
that
\begin{equation*}
\begin{split}
& \int V_i (s,\eta) f(\eta) \mu_{\rho} (d\eta) \\
&\quad \le\; \int \frac{\beta }{N^{1+\gamma}} \sum_{x \in \bb T_N} 
H_i(s,x/N)^2 \Big\{ \frac{\eta (x)}{W_x^N} + 
\frac{\eta (x+1)}{W_{x+1}^N}\Big\} f(\eta) \mu_{\rho} (d\eta)
\; +\; \frac{N^2}{\beta N^\gamma} I^W_N(f)  
\end{split}
\end{equation*}
for all density $f$ and $\beta>0$.

Recall the definition of $V_i$. Performing a simple change of
variables, we see that
\begin{eqnarray*}
\int \Big\{ \frac{\eta(x)}{W_x^N} - \frac{\eta(x+1)}{W_{x+1}^N} \Big\} \,
f\, d\mu_\rho  \;=\; \int \frac{\eta(x)}{W_{x}^N} 
\big\{ f(\eta) - f(\sigma^{x,x+1} \eta)\big\} \, d\mu_\rho\;.
\end{eqnarray*}
Since $(a-b) = (\sqrt a - \sqrt b)(\sqrt a + \sqrt b)$, by Schwarz
inequality, the previous expression is less than or equal to
\begin{eqnarray*}
\frac {N}{A} I^W_{x,x+1} (f) \;+\;
\frac A 2 \int \frac{\eta(x)}{W_{x}^N} 
\Big\{ \sqrt{f(\sigma^{x,x+1} \eta)} + \sqrt {f(\eta)} \Big\}^2
\, d\mu_\rho
\end{eqnarray*}
for all $A>0$. The same change of variables permit to estimate the
second term as
\begin{eqnarray*}
A \,  \int \Big\{ \frac{\eta(x)}{W_{x}^N}
+ \frac{\eta(x+1)}{W_{x+1}^N} \Big\} \, f(\eta)
\, d\mu_\rho\;.
\end{eqnarray*}

Choosing $A=\beta N^{-1} |H(s,x/N)|$, we obtain that for any density $f$
with respect to $\mu_{\rho}$,
\begin{eqnarray*}
\!\!\!\!\!\!\!\!\!\!\!\!\!\!\!\! &&
\int V_i (s,\eta) (s,\eta) f(\eta) \mu_{\rho} (d\eta) \;\le\;
\frac {N^2}{\beta N^{\gamma}} \sum_{x\in \bb T_N} I^W_{x,x+1}(f) \\
\!\!\!\!\!\!\!\!\!\!\!\!\!\!\!\! &&
\quad +\; \frac {\beta}{N^{1+\gamma}} \sum_{x\in \bb T_N} 
H_i(s,x/N)^2  \int \Big\{ \frac{\eta(x)}{W_{x}^N}
+ \frac{\eta(x+1)}{W_{x+1}^N} \Big\} \, f(\eta)
\, d\mu_\rho \;, 
\nonumber
\end{eqnarray*}
which proves the lemma.
\end{proof}

Recall that, by Lemma \ref{s01}, $\mc Q^*$ is concentrated on
absolutely continuous measures $\mf M = v(t,x) dt\, dx$.

\begin{corollary}
\label{s05}
Any limit point $\mc Q^*$ of the sequence $\{\mc Q_N : N\ge 1\}$
is concentrated on measures $\mf M = v(t,x) dt\, dx$ such that
\begin{equation*}
\begin{split}
E_{Q^*} \Big[ \sup_H \Big\{ \int_0^T ds\, \int_{\bb T}
& (\partial_{x} H) (s, x) v(s,x) \, dx   \\
& - 2 \int_0^T ds\, \int_{\bb T} H (s, x)^2 v(s,x) 
\, dx \Big\} \Big] \; \le \; K_0 \; .
\end{split}
\end{equation*}
In this formula the supremum is taken over all functions $H$ in
$C^{0,1}([0,T]\times \bb T)$.
\end{corollary}

\begin{proof}
Fix a limit point $\mc Q^*$ of the sequence $\mc Q_N$ and assume,
without loss of generality, that $\mc Q_N$ converges to $\mc Q^*$.
Consider a sequence $\{H_j : j\ge 1\}$ of functions in
$C^{0,1}([0,T]\times \bb T)$ dense for the uniform topology. It
follows from Lemma \ref{t4} with $\beta=1$, a summation by parts and a
coupling estimate, similar to the one used in the proof of Lemma
\ref{s01}, to replace the discrete derivative $N\{H(s, (x+1)/N) - H(s,
x/N) \}$ by the continuous one $(\partial_{x} H)(s,x/N)$, that
\begin{equation*}
  \begin{split}
E_{Q^*} \Big[ \max_{1\le i\le \ell} \Big\{ \int_0^T ds\, \int_{\bb T}
& (\partial_{x} H_i) (s, x) \, v(s,x) \, dx   \\
& - 2 \int_0^T ds\, \int_{\bb T} H_i (s, x)^2 \, v(s,x) 
\, dx \Big\} \Big] \; \le \; K_0 \;.
  \end{split}
\end{equation*}
Letting $\ell\uparrow\infty$, we conclude the proof of the lemma
applying the monotone convergence theorem.
\end{proof}

\begin{proof}[Proof of Proposition \ref{s07}.]
The proof is similar to the one of \cite[Theorem 5.7.1]{kl} and left
to the reader. Note that we have in fact
\begin{equation*}
\int_0^T ds\, \int_{\bb T} \frac{(\partial_{x} v)(s,x)^2}
{v(s,x)} \, dx \; <\; \infty\;.
\end{equation*}
\end{proof}

\subsection{$\mf M_t = \pi_t$, $W$ almost surely}

We prove in this section that $\mf M_t = \pi_t$, $W$ almost surely.

\begin{lemma}
\label{s06}
Every limit point $\mc Q^*$ of the sequence $\mc Q_N$ is concentrated
on measures $\mf M(dt,dx) = v(t,x) dt\, dx$ and $\pi(t,dx) = u(t,x)
W(dx)$ such that $u = v$ $(dt \times W(dx))$ almost surely on
$[0,T]\times \bb T$.
\end{lemma}

\begin{proof}
Fix a limit point $\mc Q^*$ of the sequence $\mc Q_N$ and assume,
without loss of generality, that $\mc Q_N$ converges to $\mc Q^*$.
Fix a continuous function $J: [0,T] \times \bb T \to \bb R$. By
Corollary \ref{s10} and by Lemma \ref{s01},
\begin{equation*}
\begin{split}
\lim_{\epsilon \to 0} 
E_{\mc Q^*} \Big[\, \Big| \int_0^T  ds\, \int_{\bb T} J(s,x) \, u(s,x)
\, W(dx) & \\
\;-\; \int_0^T  & ds\, \int_{\bb T} J(s,x) \epsilon^{-1} \, 
W_\epsilon (x) \, v(s,x) \, dx   \Big| \, \Big] \;=\;0\;.
\end{split}
\end{equation*}

It follows from the energy estimate stated in Proposition \ref{s07}
that 
\begin{equation*}
\lim_{\epsilon \to 0} 
\int_0^T  ds\, \int_{\bb T} dx\, J(s,x) \, \frac 1{\epsilon} \, 
\int_{[x,x+\epsilon)} \{ v(s,x) - v(s,y) \} \, W (dy) \;=\; 0
\end{equation*}
$\mc Q^*$ almost surely. Changing the order of summation, it follows
from the continuity of $J$ that
\begin{equation*}
\begin{split}
& \lim_{\epsilon \to 0} 
\int_0^T  ds\, \int_{\bb T} dx\, J(s,x) \, \frac 1{\epsilon} \, 
\int_{[x,x+\epsilon)}  v(s,y) \, W (dy) \\
& \quad \;=\; 
\int_0^T  ds\, \int_{\bb T}  J(s,x) \, v(s,x) \, W (dx)\;.
\end{split}
\end{equation*}
Hence, for all continuous function $J$, $\mc Q^*$ almost surely
\begin{equation*}
\int_0^T  ds\, \int_{\bb T} J(s,x) \, \{ u(s,x) - v(s,x) \}
\, W(dx)  \;=\;0\;,
\end{equation*}
which proves the lemma.
\end{proof}

\section{Uniqueness of weak solutions}
\label{ssec4}

\begin{theorem}
\label{t2}
There exists at most one weak solution of (\ref{ec2}).
\end{theorem}

\begin{proof}
We use a method due to Oleinik (cf. pg. 90 in \cite{v}). Due to the
linearity of problem (\ref{ec2}), it is enough to show that the
constant function equal to $0$ is the unique weak solution of equation
\eqref{ec2} with initial condition $u_0 \equiv 0$.

Fix such solution $u$. By condition (i), $u$ belongs to $L^2([0,T];
\mc H_{1})$. Since $u(t,\cdot)$ is continuous for almost all $t$, it
is not difficult to show that there exists a sequence of smooth
functions $u_\epsilon : [0,T]\times \bb T\to\bb R$, $\epsilon >0$,
such that $\Vert u_\epsilon\Vert_\infty \le \Vert u \Vert_\infty$ and 
\begin{equation*}
\lim_{\epsilon \to 0} \int_0^T dt \Big\{  \Vert u_\epsilon (t,\cdot) -
u (t,\cdot) \Vert^2_{2,W} \;+\; \Vert (\partial_x
u_\epsilon) (t,\cdot) - (\partial_x u) (t,\cdot) \Vert^2_2 \Big\} \;=\;
0\; .
\end{equation*}

Consider the test function $G_\epsilon:[0,T]\times \bb T \to \bb R$
defined by
\begin{equation*}
G_\epsilon(t,x) \;=\;  - \int_t^T u_\epsilon (s,x) \, ds \;.
\end{equation*}
Since $\partial_t G_\epsilon = u_\epsilon$, and since $u_\epsilon
(t,\cdot)$ converges to $u (t,\cdot)$ in $L^2(dW)$ for almost all $t$,
by the dominated convergence theorem,
\begin{equation*}
\lim_{\epsilon \to 0} \int_0^T \<\partial_t G_\epsilon (t,\cdot) 
,u(t,\cdot)\>_W \, dt \;=\;  \int_0^T \<u_t,u_t\>_W \, dt \;.
\end{equation*}
On the other hand, since 
\begin{equation*}
\int_0^T \<(\partial_x G_\epsilon) (t,\cdot) 
, (\partial_x u) (t,\cdot)\> \, dt \;=\; -\;
\int_0^T dt \, \int_t^T 
\<(\partial_x u_\epsilon) (s,\cdot)  , 
(\partial_x u) (t,\cdot)\> \, ds\;,
\end{equation*}
and since $\partial_x u_\epsilon$ converges to $\partial_x u$
in $L^2([0,T]\times \bb T)$, 
\begin{equation*}
\lim_{\epsilon \to 0} \int_0^T  \<(\partial_x G_\epsilon) (t,\cdot) 
, (\partial_x u) (t,\cdot)\> \, dt \;=\; -\;
\frac 12 \, \int_{\bb T} \Big( \int_0^T (\partial_x u) (t,x)
\, dt \Big)^2 \, dx \;.
\end{equation*}
Hence, by condition (ii), since $u_0=0$,
\begin{equation*}
\int_0^T  \<u_t,u_t\>_W \, dt \;=\; -\, \frac 14 \, \int_{\bb T} 
\Big( \int_0^T (\partial_x u) (t,x) \, dt \Big)^2 \, dx \;. 
\end{equation*}
This show that $u_t \equiv 0$ for almost every $t$, and uniqueness
follows.
\end{proof}

\section{Atomic trap models in dimension $d\ge 2$}
\label{sect1}

We prove in this section Theorems \ref{mt4}, \ref{mt5} and \ref{mt2}. 

\subsection{Capacity and trace process}

To help the reader to follow the arguments of this section, we
summarize below known results on capacity and trace processes used
later.  Consider a reversible, ergodic Markov chain $\{X_t: t\geq 0\}$
on a countable set $E$. Fix a non-empty subset $F$ of $E$ and denote
by $\{X_t^F: t \geq 0\}$ the trace process of $\{X_t: t \geq 0\}$ on
$F$, as defined in Subsection \ref{ss21}. 

Denote by $\nu$ the unique invariant probability measure of $\{X_t: t
\geq 0\}$ and by $\nu^F$ the invariant probability measure of the
trace process $\{X_t^F:t \geq 0\}$. By Lemma 5.3 of \cite{bl1},
$\nu^F$ coincides with the measure $\nu$ conditioned to $F$, and
$\nu^F$ is reversible.

For $x \in E$ (resp. $x \in F$), let $\bb P_x$ (resp. $\bb P_x^F$) be
the distribution on the path space $D(\bb R_+, E)$ (resp. $D(\bb
R_+,F)$) induced by the process $\{X_t:t \geq 0\}$ (resp. $\{X_t^F:t
\geq 0\}$) starting from $x$.

For a subset $B$ of $E$ (or $F$), denote by $H(B)$ the entry time in
$B$, defined as
\begin{equation*}
H(B) = \inf \{t \geq 0: Z_t \in B\},
\end{equation*}
where $Z_t$ stands either for $X_t$ or for $X_t^F$. The context will
always clarify to which process we are referring to. Denote by
$\tau(B)$ the time of first return of $\{X_t:t \geq 0\}$ to $B$:
\begin{equation*}
\tau(B) = \inf\{t > T_1: X_t \in B\},
\end{equation*}
where $T_1$ stands for the time of the first jump of $\{X_t:t
\geq0\}$.  When the set $B$ is a singleton $\{x\}$, we denote
$H(\{x\})$, $\tau (\{x\})$ by $H(x)$, $\tau (x)$, respectively.

Denote by $\lambda: E \to \bb R_+$ the holding times of the Markov
process $\{X_t: t\ge 0\}$. By Lemma 5.4 in \cite{bl1}, the rate
$r^F(x,y)$ at which the trace process $\{X^F_t : t\ge 0\}$ jumps from
a site $x\in F$ to a site $y\in F$, $y\not = x$, is given by
\begin{equation}
\label{f05b}
r^F(x,y) \;=\; \lambda(x) \, \bb P_x \big [
H(y) < \tau(F \setminus \{y\})\big] \;.
\end{equation}

The expectation of an entry time has a simple expression in terms of
the capacities associated to the process $\{X_t: t \geq 0\}$. Denote by
$L$ the generator of the process $\{X_t: t \geq 0\}$. Let $A, B
\subseteq E$ be two disjoint sets. Define
\begin{equation*}
\mc B(A,B) = \{f: E \to \bb R : f(x) =1 \text{ for } x \in A 
\text{ and } f(x) =0 \text{ for } x \in B\}.
\end{equation*}

Let $D$ be the Dirichlet form associated to $\{X_t :t \geq 0\}$: $D(f)
=-\int f L f d\nu$ for any $f: E \to \bb R$. The capacity $\Cap(A,B)$
between $A$ and $B$ is defined as
\begin{equation}
\label{ec1:s61}
\Cap(A,B) \;=\; \inf \{D(f): f \in \mc B(A,B)\}\;.
\end{equation}

Notice that $\Cap(A,B) = \Cap(B,A)$; it is enough to consider $\tilde
f = 1-f$. An elementary computation shows that $\Cap(A,B) =
D(f_{A,B})$, where $f_{A,B}$ is the unique solution of
\begin{equation*}
\left\{
\begin{array}{ll}
(L f)(x) =0 & \text{ if } x \in E \setminus (A \cup B)\\
f(x) = 1 & \text{ if } x \in A\\
f(x) = 0 & \text{ if } x \in B.
\end{array}
\right.
\end{equation*}
It is easy to see, by the strong Markov property, that
\begin{equation}
\label{ec2:s61}
f_{A,B} (x) = \bb P_x[H(A) < H(B)]. 
\end{equation}

Let $A$, $B$ be two disjoint subsets of $F$. Define $\Cap_F(A,B)$ as
the capacity between $A$ and $B$, with respect to the trace process
$\{X_t^F: t \geq 0\}$. By Lemma 5.4 (d) in \cite{bl1}, 
\begin{equation}
\label{f04}
\Cap_{F}(A,B) \;=\; \frac {\Cap(A,B)}{\nu(F)}\;\cdot
\end{equation}

The first result of this section establishes the relation between
capacity and expectation of hitting times.

\begin{lemma}
\label{s02b}
For any subset $A$ of $F$ and any $y$ in $F\setminus A$,
\begin{equation*}
\bb E^{F}_y \big[H (A) \big]
\;=\; \frac 1{\Cap (y,A)} \sum_{z\in F} \nu (z)\,
\bb P_z \big[H(y) < H (A) \big]\;.
\end{equation*}
\end{lemma}

\begin{proof}
Fix $A\subset F$ and $y$ in $F\setminus A$. By equation (4.12) in
\cite{bl1},
\begin{equation*}
\bb E^{F}_y \big[ H (A) \big]
\;=\; \frac 1{\Cap_{F}(y,A)} \sum_{z\in F} \nu^{F}(z)\,
\bb P^{F}_z \big[ H(y) < H (A) \big]\;.
\end{equation*}
To prove the lemma, it remains to recall that $\nu^{F}$ is the measure
$\nu$ conditioned to $F$, \eqref{f04} and Lemma 5.4 (a) in \cite{bl1}.
\end{proof}

Note that $\bb P_z [ H(y) < H (A) ]$ vanishes for $z$ in
$A$. In particular,
\begin{equation}
\label{f06}
\bb E^{F}_y \big[H (A) \big]
\;\le\; \frac {\nu(F\setminus A)}{\Cap (y,A)} \;\cdot
\end{equation}

Capacities are also related to return times. Next result follows from
equation (6.10) in \cite{bl1}.

\begin{lemma}
\label{escape}
Let $A$ be a finite subset of $E$, and let $y\in E \setminus A$. Then,
\begin{equation*}
\bb P_y[ H(A) < \tau(y)] \;=\; \frac{\Cap(A,\{y\})}{\lambda(y)\nu(y)}
\; \cdot
\end{equation*}
\end{lemma} 

When the Markov chain $\{X_t: t \geq 0\}$ is not ergodic, the definition
of the capacity can be generalized in a natural way. Assume that there
exists a positive measure $\nu$, reversible and invariant for $\{X_t:
t\geq 0\}$. In this case, of course, $\nu(E)=+\infty$.

To define the capacity between a finite set $A$ and infinity, consider
an increasing sequence of finite sets $B_n \subseteq E$ such that
$\cup_n B_n =E$. Since $A$ is finite, $\Cap(A,B_n^c)$, given by the
variational formula \eqref{ec1:s61}, is well defined for $n$ large
enough. The sequence of functions $f_{A,B_n^c}$, introduced in
\eqref{ec2:s61}, is increasing and bounded. Therefore, we can define
$f_A(x) =\lim_n f_{A,B_n^c}(x)$. It is not difficult to check that
\begin{equation*}
f_{A} (x) \;=\; \bb P_x[H(A) < \infty] \;, \quad
D(f_A) \;=\; \inf \{D(f): f\in \mc B(A)  \}\;,
\end{equation*}
where $\mc B(A)$ is the set of finitely supported functions $f:E\to
\bb R$ such that $f(x)=1$ for $x\in A$. Let $\Cap(A) := D(f_A)$ be the
capacity of $A$ with respect to infinity.

By the dominated convergence theorem and Lemma \ref{escape}, the
following result holds.

\begin{lemma}
\label{transient-escape}
For any $y$ in $E$,
\begin{equation*}
\bb P_y[\tau(y)=+\infty] \;=\; \frac{\Cap(y)}{\lambda(y) \, \nu(y)}\;
\cdot 
\end{equation*}
\end{lemma}

We conclude this subsection with two estimates for the simple
symmetric random walk on the torus $\bb T^d_N$ or on the lattice $\bb
Z^d$. Taking adavantage of the commuting time identity, which relates
expectation of hitting times with capacities, Proposition 10.13 of
\cite{peres} establishes the following bounds for the hitting times of
the simple symmetric random walk on $\bb T^d_N$: 

\begin{lemma}
\label{captorus}
Let $x$, $y$ two points at distance $k$ on the torus $\bb
T^d_N$. There exist constants $0<c_d<C_d<+\infty$ such that
in dimension $d\ge 3$,
\begin{equation*}
c_d N^d \;\leq\; \bb E_x[H(y)] \;\leq\; C_d N^d \quad
\text{uniformly in $k$}
\end{equation*}
and in dimension $2$,
\begin{equation*}
c_2 N^2 \log k \;\leq\; \bb E_x[H(y)] \;\leq\; C_2 N^2 \log (k+1) \;.
\end{equation*}
\end{lemma}

For $N >0$, let us denote by $\Lambda_N^*$ the cube of length $2N+1$
centered at the origin: $\Lambda_N^* = \{-N,\dots,N\}^d$. Denote by
$\delta \Lambda_N$ its inner boundary: $\delta \Lambda_N = \{x \in
\Lambda_N^* : \exists y \notin \Lambda_N^* \text{ with } |y-x|=1\}$.
The following lemma is Proposition 2.2.2 of \cite{lawler}:

\begin{lemma} 
\label{prop2.2.2}
Let $X(t)$ be the simple symmetric random walk in $\bb Z^d$, $d \geq
3$. Then, there exist constants $0<c_d<C_d <+\infty$ such that for any
set $A \subseteq \Lambda_N^*$ and any $x \in \partial \Lambda_{2N}^*$
we have
\begin{equation*}
c_d \, N^{2-d} \, \Cap(A) \;\leq\; \bb P_x[H_A <+\infty] 
\;\leq\; C_d \, N^{2-d} \, \Cap(A)\; . 
\end{equation*}
\end{lemma}

Note that $\Cap (A)$ is finite because the random walk is transient.

\subsection{Random walks in $\mathbb{T}^d_N$, $d \geq 3$}

In this subsection we prove some properties of the simple random walk
on $\mathbb{T}^d_N$ which will be used to establish its metastable
behavior. Denote by $\{Y^N_k : k\ge 0\}$ the discrete time,
nearest-neighbor, symmetric random walk on $\bb T^d_N$ and let $\bb
Q^N_x$, $x\in \bb T^d_N$, be the measure on $D(\bb Z_+, \bb T^d_N)$
induced by the random walk $Y^N$ starting from $x$. Expectation with
respect to $\bb Q^N_x$ is denoted by the same symbol.

For a subset $B$ of $\bb T^d_N$, denote by $H(B)$ the entry time in
$B$, defined as
\begin{equation*}
H(B) \;=\; \inf\{k\ge 0 : Y^N_k\in B\}\;.
\end{equation*}

In Lemma~\ref{aug2} below we prove that whenever the simple random
walk starts from a point isolated from the very deep traps,
asymptotically, the next very deep trap to be visited is uniformly
chosen.  In Corollary~\ref{aug1} we obtain the limiting distribution
of the next very deep trap to be visited starting from another very
deep trap. Corollary~\ref{aug3} presents the limit of the capacity
between two points of $\mathbb{T}^d_N$ far apart.

Let $d(x,y)$ be the distance induced by the graph $\bb T^d_N$. For a
subset $\Gamma$ of $\bb T^d$ and $r>0$, denote by $B(\Gamma, r)$ the
set of sites in $\bb T^d_N$ at distance less than or equal to $r$ from
$\Gamma$: $B(\Gamma, r) = \{x\in \bb T^d_N : d(x,\Gamma)\le r\}$; and
denote by $\partial G$ the sites not in $\Gamma$ which are at distance
one from $\Gamma$: $\partial \Gamma = \{x\not\in\Gamma :
d(x,\Gamma)=1\}$. In these definitions we identified $\Gamma$ with its
immersion in $\bb T^d_N$: $\Gamma = \Gamma\cap \bb T^d_N$.

\begin{lemma}
\label{aug2}
Suppose $l_N \uparrow \infty$ and $A^N_M = \{x_1^N,\dots, x_M^N\}
\subset \mathbb{T}^d_N$ are such that, if $i\neq j$, $d(x_i^N,x_j^N)
\geq l_N$. Then, defining $A^N_{M,i}$ to be $A^N_M \setminus
\{x_i^N\}$, we have
\begin{equation*}
\lim_{N\to\infty}
\sup_{y\in B(A^N_M,l_N)^c} \left| \mathbb{Q}^N_{y} \big[ H(x_1^N) 
< H(A^N_{M,1}) \big]  - \frac{1}{M} \right| \;=\; 0 \;. 
\end{equation*}
\end{lemma}

\begin{proof}
The proof is based on the fact that a site is reached on the scale
$N^d$ and equilibrium is reached on the scale $N^2$. Hence, in
an intermediate scale, the process has not reached $A^N_M$ and is in
equilibrium. In particular, it has a probability $1/M$ to attain
$x^N_1$ before the set $A^N_{M,1}$.

First, we prove that the process does not reach a site in the scale
$N^{5/2}$:
\begin{equation}
\label{nohit}
\lim_{N \rightarrow \infty} \sup_{y; d(y,0) \geq l_N}
\mathbb{Q}^N_y \big[ H(0) < N^{5/2} \big] \;=\; 0\;. 
\end{equation}
By the strong Markov property,
\begin{equation*}
\begin{split}
\mathbb{Q}^N_y[H(0) < N^{5/2}] & \;\leq\; \mathbb{Q}^N_y[H(0) <
H(\partial B(0,N/8))] \\ 
& \;+\; \sup_{z \in \partial B(0,N/8)}\mathbb{Q}^N_z[H(0) < N^{5/2}]\;.
\end{split}
\end{equation*}
Since $d(y,0) \geq l_N$, by Lemma \ref{prop2.2.2}, the first term in
the sum above is bounded by $C_0/l_N^{d-2}$ for some finite constant
$C_0$ independent of $N$. We can therefore suppose in \eqref{nohit}
that $l_N = N/8$.

Denote by $\{R_k : k\ge 1\}$ and $\{D_k : k\ge 1\}$ the successive
return and departure times between $B_1 = B(0,N/8)$ and $B_2 = B(0,N/4)$:
\begin{equation*}
\begin{aligned}
R_1 &= H_{B_1} & D_1 & = R_1 + H_{B_2^c} \circ \theta_{R_1}\\
R_n &= D_{n-1} + H_{B_1} \circ \theta_{D_{n-1}} & D_n & = R_n + H_{B_2^c}
\circ \theta_{R_n} \;, \; n\ge 2\;.
\end{aligned}
\end{equation*}
Here $\theta_s : D(\mathbb{Z}_+, \mathbb{T}^d_N) \rightarrow D(\mathbb{Z}_+,
\mathbb{T}^d_N)$ is the shift map given by $Y_t \circ \theta_s = Y_{s + t}$.

For $y$ such that $d(y,0) \geq N/8$, by the strong Markov property,
\begin{equation}
\label{hit52}
\begin{split}
\mathbb{Q}^N_y[H(0) < N^{5/2}] & \;\leq\; \mathbb{Q}^N_y[R_{N^{d-9/4}}
\leq N^{5/2}]  \;+\; \mathbb{Q}^N_y[ H(0) < R_{N^{d-9/4}} ] \\
& \;\leq\; \sup_{z; d(z,0) \geq N/8} \mathbb{Q}^N_z[R_{N^{d-9/4}}  
\leq N^{5/2}]\\
& \;+\;  N^{d-9/4} \sup_{z \in \partial B(0,N/8)}
\mathbb{Q}^N_z[H(0) < H(\partial B(0,N/4))] \;.
\end{split}
\end{equation}
The right hand side does not depend on the choice of $y$ and tends to
zero as $N \uparrow \infty$, by \cite{benjamini}, Proposition~1.1 with
$u = N^{-1/2}$, and Lemma \ref{prop2.2.2}. This proves (\ref{nohit}).

In the time scale $N^{5/2}$ the process reaches equilibrium.  More
precisely, denote by $\pi^N$ the uniform probability measure on $\bb
T^d_N$ and by $\Vert \mu - \nu\Vert$ the total variation distance
between two measures $\mu$, $\nu$ on $\bb T^d_N$.  By Corollary 5.3
and equation (5.9) in \cite{peres}, for an arbitrary
sequence $y^N \in B(A^N_M, l_N)^c$, 
\begin{equation}
\label{baba}
\lim_{N\to\infty}
\Big\lVert \mathbb{Q}^N_{y^N}[Y_{N^{5/2}} = \cdot] - \pi^N(\cdot)
\Big\rVert \;=\; 0\;.
\end{equation}

In particular, for $1\le i\le M$ and for an arbitrary sequence $y^N$
in $B(A^N_M, l_N)^c$,
\begin{equation}
\label{fl01}
\lim_{N\to\infty} \frac 1{N^d} \Big| \, \bb Q^N_{y^N} [H(x^N_i)] - \bb
Q^N_{\pi^N}[H(x^N_i)]  \,\Big| \;=\;0\;.
\end{equation}
To prove this claim, fix $1\le i\le M$ and introduce the indicators of
the sets $H(x^N_i) < N^{5/2}$, $H(x^N_i) \ge N^{5/2}$, to obtain that
\begin{equation*}
\bb Q^N_{y^N} [H(x^N_i)] \;=\;
\bb Q^N_{y^N} \Big[\, \bb Q^N_{Y_{N^{5/2}}}[H(x^N_i)]\,\Big] \;+\; R_N\;, 
\end{equation*}
where the remainder $R_N$ is absolutely bounded by
\begin{equation*}
N^{5/2}\; +\; \sup_{z\in\bb T^d_N} \bb Q^N_{z} [H(x^N_i)]
\, \mathbb{Q}^N_{y^N}[H(x^N_i) < N^{5/2}]\;.
\end{equation*}
Hence,
\begin{equation*}
\bb Q^N_{y^N} [H(x^N_i)] - \bb Q^N_{\pi^N}[H(x^N_i)] 
\;=\; \bb Q^N_{y^N} \Big[ \bb Q^N_{Y_{N^{5/2}}}[H(x^N_i)] 
- \bb Q^N_{\pi^N}[H(x^N_i)]  \, \Big] \;+\; R_N
\end{equation*}
is absolutely bounded by
\begin{equation*}
N^{5/2} \;+\; \sup_{z\in\bb T^d_N} \bb Q^N_{z} [H(x^N_i)]
\Big\{ \mathbb{Q}^N_{y^N}[H(x^N_i) < N^{5/2}] \;+\;
\Big \lVert \mathbb{Q}^N_{y^N}[Y_{N^{5/2}} = \cdot] - 
\pi^N(\cdot) \Big\rVert \Big\} \;.
\end{equation*}
By \eqref{nohit}, \eqref{baba} and Lemma \ref{captorus}, this
expression divided by $N^d$ vanishes as $N\uparrow\infty$. This proves
\eqref{fl01}. 

Recall that $v_d$ stands for the probability that a symmetric,
nearest-neighbor random walk on $\bb Z^d$ never returns to its
starting point. By the estimate on the expected hitting time (3.2) of
\cite{david2}, it follows from \eqref{fl01} that
\begin{eqnarray}
\label{f18}
\!\!\!\!\!\!\!\!\!\!\!\!\!\!
&& \lim_{N\to\infty}  \frac{1}{N^d} \bb Q^N_{y^N}[H(x^N_i)] \;=\;
\frac{1}{v_d}, \text { and, analogously, for $j\not =1$, } \\ 
\!\!\!\!\!\!\!\!\!\!\!\!\!\!
&& \quad \lim_{N\to\infty} \frac{1}{N^d} \bb Q^N_{x^N_j}[H(x^N_1)] \;=\;
\frac{1}{v_d} \text{ and } \lim_{N\to\infty} \frac{1}{N^d}
\bb Q^N_{x_1^N}[H(A^N_{M,1})] \;=\; \frac{1}{v_d(M-1)}\; ,
\nonumber
\end{eqnarray}
because, by Lemma \ref{transient-escape}, the escape probability $v_d$
equals the capacity between the origin and infinity in $\mathbb{Z}^d$.
Note that $\lambda(0)=1$ and that the capacity is computed with
respect to the counting measure. 

Define $S$ to be the stopping time given by the first visit to $y^N$
after the first visit to $x_1^N$, after visiting $A^N_{M,1}$. By the
strong Markov property, $\bb Q^N_{y^N}[S]$ is equal to
\begin{equation*}
\bb Q^N_{y^N} [H(A^N_{M,1})] \;+\; \sum_{j \neq 1}
\mathbb{Q}^N_{y^N}[Y_{H(A^N_{M,1})}=x_j^N] \, 
\bb Q^N_{x^N_j}[H(x_1^N)] \;+\; \bb Q^N_{x^N_1}[H(y^N)] \; .
\end{equation*}
Hence, by \eqref{f18}, 
\begin{equation}
\label{f19}
\lim_{N\to\infty}  \frac{1}{N^d} \bb Q^N_{y^N}[S] \;=\;
\frac 1 {v_d} \Big( \frac{1}{M-1} + 2 \Big)\;.
\end{equation}

We are now in a position to prove the lemma. Fix an arbitrary sequence
$y^N$ in $B(A^N_M, l_N)^c$.  Since $\bb Q^N_{y^N}[\#\{\text{visits to
  $x_1^N$ before $H(A^N_{M,1})$}\}] = \bb Q^N_{y^N}[\#\{\text{visits to
  $x_1^N$}$ before $H(A^N_{M,1})\} \, \mb 1\{H(x_1^N) <
H(A^N_{M,1})\}]$, by the strong Markov property,
\begin{equation}
\label{f17}
\mathbb{Q}^N_{y^N} [H(x_1^N) < H(A^N_{M,1})] 
\;=\; \frac{\bb Q^N_{y^N}[\#\{\text{visits to $x_1^N$ before
    $H(A^N_{M,1})$}\}]}
{\bb Q^N_{x_1^N}[\#\{\text{visits to $x_1^N$ before $H(A^N_{M,1})$}\}]} \; \cdot
\end{equation}
To estimate the numerator, observe that, by \cite{aldous}, Chapter~2,
Proposition~3 and Lemma~7,
\begin{eqnarray*}
\!\!\!\!\!\!\!\!\!\!\!\!\!\!\!\!\!
&& \frac{1}{N^d} \bb Q^N_{y^N}[S] \;=\;
\bb Q^N_{y^N} \big[ \#\{\text{visits to $x_1^N$ before $S$}\} \big] \\
\!\!\!\!\!\!\!\!\!\!\!\!\!\!\!\!\!
&& =\; \bb Q^N_{y^N} \big [\#\{\text{vis. to $x_1^N$ bef.
  $H(A^N_{M,1})$}\} \big] \;+\; \bb Q^N_{x_1^N} \big [\#\{\text{vis. to
  $x_1^N$ bef. $H(y^N)$}\} \big] \\ 
\!\!\!\!\!\!\!\!\!\!\!\!\!\!\!\!\!
&& =\; \bb Q^N_{y^N} \big[\#\{\text{vis. to $x_1^N$ bef.
  $H(A^N_{M,1})$}\} \big] \;+\; \frac{1}{N^d}\Big\{\bb Q^N _{x_1^N}[H(y^N)]
+ \bb Q^N_{y^N}[H(x_1^N)] \Big\}\;.  
\end{eqnarray*}
Hence, the right hand side of \eqref{f17} can be written as
\begin{equation*}
\frac{N^{-d} \Big( \bb Q^N_{y^N}[S] - \bb Q^N_{x_1^N}[H(y^N)] -
\bb Q^N_{y^N}[H(x_1^N)]\Big)}
{\bb Q^N_{x_1^N}[\#\{\text{visits to $x_1^N$ before $H(x_1^N)\circ 
\theta_{H(A^N_{M,1})} + H(A^N_{M,1})$}\}]}\; \cdot
\end{equation*}
Let $S'$ be the stopping time $H(A^N_{M,1}) + H(x_1^N)\circ
\theta_{H(A^N_{M,1})}$. By Lemma~7, Chapter~2 in \cite{aldous}, the
denominator is equal to $N^{-d} \bb Q^N_{x_1^N}[S']$. Hence, by the
strong Markov property on $H(A^N_{M,1})$, the previous ratio is equal
to
\begin{equation*}
\frac{ N^{-d} \Big( \bb Q^N_{y^N}[S] - \bb Q^N_{x_1^N}[H(y^N)] -
\bb Q^N_{y^N}[H(x_1^N)]\Big)} 
{ N^{-d} \Big( \bb Q^N_{x_1^N}[H(A^N_{M,1})] + {\displaystyle 
\sum_{j \neq 1}} \mathbb{Q}^N_{x_1^N}[Y_{H(A^N_{M,1})} = x_j^N]
\, \bb Q^N_{x_j^N}[H(x_1^N)] \Big)} \; \cdot
\end{equation*}
By \eqref{f18}, \eqref{f19}, this expression converges to $M^{-1}$ as
$N\uparrow\infty$.  Since the sequence $y^N \in B(A^N, l_N)^c$ is
arbitrary, we are done.
\end{proof}

For a subset $F$ of $\bb T^d_N$, let $\hat \tau(F)$ be the return time
to $F$ for the discrete time random walk $\{Y^N_k : k\ge 0\}$:
\begin{equation*}
\hat \tau(F) \;=\; \inf\{k\ge 1 : Y^N_k \in F\}\;.
\end{equation*}
For subsets $A$, $B$ of $\bb T^d_N$ such that $A\cap B = \phi$, let
$\Cap_{Y^N} (A,B)$ be the capacity between $A$ and $B$ induced by the
process $Y^N$:
\begin{eqnarray*}
\Cap_{Y^N} (A,B) \; =\; \inf_f \frac 1{4d} \sum_{x\in \bb T^d_N}
\sum_{y\sim x} [f(y) - f(x)]^2\;,
\end{eqnarray*}
where the infimum is carried over all functions $f: \bb T^d_N \to
\bb R$ such that $f(x) =1$ for all $x$ in $A$, $f(x)=0$ for all $x$ in
$B$.

\begin{corollary}
\label{aug1}
Under the same conditions of Lemma~\ref{aug2}, for $j \neq 1$,
\begin{equation*}
\lim_{N\to\infty} \mathbb{Q}^N_{x_j^N}[H(x_1^N) < 
\hat \tau(A^N_{M,1})] \;=\;  \frac{v_d}{M}\; \cdot 
\end{equation*}
\end{corollary}

\begin{proof}
Since
\begin{equation*}
\lim_{N\to\infty} \mathbb{Q}^N_{x^N_j}[H(\partial B(x^N_j,l_N)) < 
\hat \tau(x^N_j)]  \;=\; v_d\;, 
\end{equation*}
the result follows from the strong Markov property and
Lemma~\ref{aug2}.
\end{proof}

\begin{corollary}
\label{aug3}
If $l_N \uparrow \infty$ and $x^N, y^N \in \mathbb{T}^d_N$, are such
that $d(x^N, y^N) \geq l_N$, 
\begin{equation}
\lim_{N\to\infty} \Cap_{Y^N}(x^N,y^N) \;=\; \frac{v_d}{2}\;\cdot
\end{equation}
\end{corollary}

\begin{proof}
The corollary is a direct application of Lemma \ref{escape} and
Corollary~\ref{aug1}. 
\end{proof}

\subsection{Random walk on $\mathbb{T}^d_N$ for $d \geqslant 2$}

We present similar results to the ones stated in the previous
subsection, but which also hold in dimension $2$. Here, however, we
need to impose that the distances $l_N$, between the very deep
traps and the starting point of the walk, grow close to linearly.

Although Lemma~\ref{aug5} below also holds for $d \geqslant 3$ and
could be used in place of Lemma~\ref{aug2}, we keep both results since
they are based on different arguments.  Let $l_N$ be an increasing
sequence such that $l_N/N \rightarrow 0$ and $l_N/N^\alpha \rightarrow
\infty$ for every $\alpha < 1$. In this way,
\begin{equation}
\label{fc02}
\lim_{N\to\infty} \frac{l_N}N \;=\; 0\;, \quad
\lim_{N\to\infty} \frac{\log l_N}{\log N} \;=\; 1 \;. 
\end{equation}

In this subsection it will be more convenient to work with the distance
$d_2(x, y)$ in $\mathbb{T}^d_N$ given by $N$ times the Euclidean distance
between $x$ and $y$ in $\mathbb{T}^d$.

\begin{lemma}
\label{aug5}
Consider a sequence of sets $A^N_M = \{x_1^N, \dots, x_M^N\} \subset
\mathbb{T}^d_N$ such that $d_2(x_i,x_j) \geqslant l_N$ {\rm (}$0
\leqslant i < j \leqslant M${\rm )} for some sequence $\{l_N : N\ge
1\}$ satisfying \eqref{fc02}. Then,
\begin{equation*}
\lim_{N \rightarrow \infty} \sup_{y \in B(A_M^N,l_N)^c} \left|
\, \mathbb{Q}^N_{y} \Big[H(x_1^N) < H(A_{M,1}^N) \Big] - \frac{1}{M}
\right| \;=\; 0\;. 
\end{equation*}
\end{lemma}

Since the result only concerns the first point in $A_M^N$ to be
visited, we can suppose that the process $\{Y_k : k\ge 0\}$ is a lazy
random walk in $\mathbb{T}^d_N$, i.e. with probability one half $Y$
does not move, otherwise it jumps uniformly to one of its neighbors.
Before going into the proof of Lemma~\ref{aug5}, we collect some
properties of the hitting and mixing times of $Y$.

Recall that $|| \cdot ||$ denotes the total variation distance between
two probability measures. The following bound on the mixing time on
the torus follows, for instance, from Corollary 5.3 and equation (5.9)
in \cite{peres}.
\begin{equation}
\label{mixtime}
 \lim_{\beta \rightarrow \infty} \limsup_{N\rightarrow \infty} \big
 \lVert \mathbb{Q}^N_y [Y_{\beta N^2} = \cdot ] - \pi^N(\cdot) \big
 \rVert = 0\;. 
\end{equation}
Of course, the same result holds for any sequence $\{t_N:N\ge 1\}$
which increases to $\infty$ faster than $N^2$.

We claim that for every $\beta > 0$ and $x^N, y^N \in
\mathbb{T}^d_N$ such that $d_2(x^N,y^N) \geqslant l_N$ 
\begin{equation}
\label{nohitN2}
\lim_{N\to\infty} \mathbb{Q}^N_{y^N} [H(x^N) \leqslant \beta N^2] 
\;=\; 0\;. 
\end{equation}
Since for $d \geqslant 3$ this statement follows from (\ref{nohit}),
we concentrate in the case $d = 2$. Recall that $\mathbb{Y}$ stands
for the simple random walk on $\mathbb{Z}^2$ (with law $\mathbb{P}$),
and denote by $\phi_N:\mathbb{Z}^2 \rightarrow \mathbb{T}^2_N$ the
canonical projection. In view of the invariance principle in
$\mathbb{Z}^2$, it is enough to prove that for every $R \geqslant 0$,
\begin{equation*}
\lim_{N\to\infty}
 \mathbb{P}_{\bar y^N} \big[\phi_N(\mathbb{Y}) \text{ visits $x^N$
   before $\mathbb{Y}$ exits $B(\bar y^N, R N)$}\big] \;=\; 0\;,
\end{equation*}
where $\bar y^N \in \mathbb{Z}^2$ is such that $\phi_N(\bar y^N) = y^N
\in \mathbb{T}^2_N$.  This follows, for instance, from \cite{bcm1},
(225) since for every point $\bar x^N \in B(\bar y^N, R N)$ such that
$\phi(\bar x^N) = x^N$, we have $|\bar y^N - \bar x^N| \geqslant l_N$
and, moreover, the number of possible choices for $\bar x^N$ is
bounded uniformly on $N \geqslant 1$. This establishes (\ref{nohitN2})
for $d = 2$.

A straightforward consequence of \eqref{nohitN2} is that
\begin{equation}
\label{fc10}
\lim_{N\to\infty} \mathbb{Q}^N_{\pi^N} [H(x^N) \leqslant \beta N^2] 
\;=\; 0 
\end{equation}
for any sequence $\{x^N : N\ge 1\}$ in $\bb T^d_N$.

Now we can state the hitting time estimate we will use during the
proof. Consider the scales
\begin{equation*}
h^d_N = \begin{cases}
N^2 \log N & \text{if $d = 2$},\\
N^d & \text{if $d \geqslant 3$}\; .
\end{cases}
\end{equation*}
We claim that for every $d \geqslant 2$, 
\begin{equation}
\label{nohit2}
\lim_{\gamma \rightarrow 0} \limsup_{N \rightarrow \infty} \,\,\,
\sup_{x,y; d_2(x,y) > l_N} \,\, \mathbb{Q}^N_y[H(x) < \gamma h^d_N] =
0 \; . 
\end{equation}

Indeed, fix $\delta > 0$. By (\ref{nohitN2}), for any $\beta>0$, for
$N$ sufficiently large, and for any $x$, $y$ such that $d_2(x,y) >
l_N$,
\begin{equation*}
\mathbb{Q}^N_y[H(x) < \gamma h^d_N] 
\;\le\; \delta \;+\; \mathbb{Q}^N_y 
\big [  H(x) \ge \beta N^2 \,,\, H(x) < \gamma h^d_N \big]\;.
\end{equation*}
On the set $H(x) \ge \beta N^2$, $H(x) = \beta N^2 + H(x) \circ
\theta_{\beta N^2}$. Therefore, by the Markov property at $\beta N^2$,
the second term is less than or equal to
\begin{equation*}
E_{\mathbb{Q}^N_y} \Big[ \mathbb{Q}^N_{Y_{\beta N^2}} 
\big [  H(x) < \gamma h^d_N \big] \, \Big]\;.
\end{equation*}
By (\ref{mixtime}), for $\beta$ large enough, this expression is
bounded by
\begin{equation*}
\delta \;+\; \mathbb{Q}^N_{\pi^N} \big[ H(x) < \gamma h_N^d \big] \;. 
\end{equation*}
The result follows now from Theorem~2.1 of \cite{david2} and Lemma
\ref{captorus}.

Finally, we claim that for every positive $\gamma$, the probability to
hit a very deep trap before $\gamma h^d_N$ is bounded away from zero
in $N$. More precisely, for every $\gamma >0$,
\begin{equation}
\label{fc11}
\limsup_{N \to \infty} \mathbb{Q}^N_{\pi^N} \big[ H(x) \geq \gamma h_N^d
\big] \;<\; \;1.
\end{equation} 
The claim above also follows from Theorem~2.1 of \cite{david2} and
Lemma \ref{captorus}.

\begin{proof}[Proof of Lemma~\ref{aug5}] 
  Our strategy is to consider several consecutive attempts to hit one
  of the points in $A_M^N$. For $\gamma >0$ and a positive integer
  $L$, define the times
\begin{equation*}
a_i \;=\; i \, ( L N^2 + \gamma \, h_N^d ) \;, \quad 
b_i \;=\; i \, ( L N^2 + \gamma \, h_N^d ) \;+\;  L N^2
\end{equation*}
for $i \ge 0$. Intuitively, for each $i$, we use the intervals $[a_i,
b_i]$ to approach equilibrium measure $\pi^N$ and the intervals $[b_i,
a_{i + 1}]$ to attempt to hit the set $A_M^N$.

Let
\begin{equation*}
\begin{split}
& R_{L,N} \;:=\; \max_{y\in \bb T^d_N} \big\Vert \, 
\mathbb{Q}^N_y[Y^N_{LN^2} = \cdot] - \pi^N(\cdot) \,\big\Vert\; , \\
& \qquad S_{L,N} \;:=\;  \sup_{y \in B(A_M^N, l_N)^c}
\mathbb{Q}^N_y[H(A_M^N) \le LN^2]\;, \\
& \qquad\qquad T_{L,N} \;:=\; \mathbb{Q}^N_{\pi^N} [ H(A^N_M) \le LN^2]\;.
\end{split}
\end{equation*}
By symmetry, the maximum is irrelevant in the definition of $R_{L,N}$,
and, by \eqref{mixtime}, $R_{L,N}$ vanishes as $N\uparrow\infty$ and
then $L\uparrow\infty$. By \eqref{nohitN2}, \eqref{fc10}, $S_{L,N}$ an
$T_{L,N}$ vanish as $N\uparrow\infty$ for every $L\ge 1$.

For $0\le s < t$, define the random variable $J_{s,t}$, which takes
the value $0$ if the set $A_M^N = \{x_1^N, \dots, x_M^N\}$ is not
visited between the times $s$ and $t$ and otherwise $J_{s,t} =
1,\dots,M$ according to the index of the first point in $A_M^N$
visited in this interval.

The proof of the lemma is divided in two parts. We first claim that
for every $\gamma>0$,
\begin{equation}
\label{fc07} 
\limsup_{N\to\infty}\sup_{y \in B(A_M^N, l_N)^c} 
\Bigg| \mathbb{Q}^N_{y} \Big[H(x_1^N) < H(A_{M,1}^N) \Big] -\; 
\frac{\bb Q^N_{\pi^N} \big [J_{b_0,a_1} =1 \big]}
{\bb Q^N_{\pi^N} \big [J_{b_0,a_1} \not = 0 \big]} \Bigg| \;=\; 0\; . 
\end{equation}
Note that this expression does not depend on $L$, but only on $\gamma$
and $N$. Then, we prove that
\begin{equation}
\label{fc08}
\lim_{\gamma \to 0} \limsup_{N\to\infty} \Bigg| \frac{\bb Q^N_{\pi^N}
\big [J_{b_0,a_1} =1 \big]} {\bb Q^N_{\pi^N} \big [J_{b_0,a_1}
\not = 0 \big]} - \frac 1M \Bigg| = 0 \;.
\end{equation}
Clearly, Lemma \ref{aug5} follows from \eqref{fc07} and \eqref{fc08}.

The proof of \eqref{fc07} relies on three estimates. Consider the
event $\mf D_F = [J_{a_i,b_i} \neq 0 \text{ for some } i = 0, \dots, F
- 1]$, $F\ge 1$. This event indicates that some very deep trap is
visited in one of the ``mixing'' intervals $[a_j,b_j]$. We claim that
\begin{equation}
\label{fc01}
\sup_{y \in B(A_M^N, l_N)^c} \mathbb{Q}^N_y [\mf D_F] \;\le \; 
F \, T_{L,N} \;+\; S_{L,N} \;+\; R_{L,N} \;.
\end{equation}

Fix a site $y$ in $B(A_M^N, l_N)^c$ and decompose the event $\mf D_F$
according to whether the set $A^N_M$ has been attained before time
$LN^2$ or not to get that
\begin{equation*}
\mathbb{Q}^N_y [\mf D_F] \;\leqslant\; \mathbb{Q}^N_y[H(A_M^N) \le LN^2] \;+\;
\mathbb{Q}^N_y [\mf D_F\,,\, H(A_M^N) > LN^2]\;.
\end{equation*}
The first term is bounded by $S_{L,N}$. On the set $H(A_M^N) > LN^2$,
the event $\mf D_F$ is equal to $\cup_{1\le i\le F-1} \{J_{a_i,b_i}
\neq 0\}$. Hence, by the Markov property at time $b_0$, the second
term on the right hand side of the previous inequality is bounded by
\begin{equation*}
E_{\mathbb{Q}^N_y} \Big [ \mathbb{Q}^N_{Y^N_{b_0}} \Big[ 
\bigcup_{i=1}^{F-1} \{J_{a_i-b_0,b_i-b_0} \neq 0\} \Big]\,\Big] \;
\le\; R_{L,N} \;+\; \mathbb{Q}^N_{\pi^N} \Big[ \bigcup_{i=0}^{F-2} 
\{J_{a_i,b_i} \neq 0\} \Big] \;.
\end{equation*}
Since $\pi^N$ is the stationary state, the second term is smaller than or
equal to $(F-1)$ $\mathbb{Q}^N_{\pi^N} [ H(A^N_N) \le LN^2]$. In
conclusion, we proved that
\begin{equation*}
\mathbb{Q}^N_y [\mf D_F] \;\leqslant\; S_{L,N} \;+\; 
F \, \mathbb{Q}^N_{\pi^N} [ H(A^N_M) \le LN^2] \;+\; R_{L,N}\;,
\end{equation*}
which is exactly \eqref{fc01}.

Let $\mf E_F$, $F\ge 1$, be the event that no site in $A^N_M$ has been
visited in the ``hitting'' intervals $[b_i,a_{i+1}]$, $0\le i\le F-1$:
$\mf E_F = [J_{b_i,a_{i+1}} = 0 \text{ for all } i = 0, \dots, F -
1]$. We claim that
\begin{equation}
\label{fc05}
\sup_{y \in B(A_M^N, l_N)^c} \Big| \bb Q^N_y[\mf E_F] 
\;-\; \bb Q^N_{\pi^N} \big [J_{b_0,a_1} =0 \big]^F \Big| 
\;\le\; F R_{L,N}\;.
\end{equation}

Fix a site $y$ in $B(A_M^N, l_N)^c$. By the Markov property, 
\begin{equation*}
\bb Q^N_y[\mf E_F] \;=\;
\bb Q^N_y \Big[ \bigcap_{i=0}^{F-2} \{J_{b_i,a_{i+1}}
= 0\} \; \bb Q^N_{Y^N_{a_{F-1}}} \big[J_{b_0,a_1} =0 \big] \, \Big]\;.
\end{equation*}
Applying the Markov property at time $b_0 = LN^2$, this expression
can be written as
\begin{equation*}
\bb Q^N_y \Big[ \bigcap_{i=0}^{F-2} \{J_{b_i,a_{i+1}}
= 0\} \Big] \bb Q^N_{\pi^N} \big [J_{0,a_1-b_0} =0 \big]
\;+\; R_N
\end{equation*}
where the remainder $R_N$ is absolutely bounded by $R_{L,N}$. Since
$\pi^N$ is the stationary state, $\bb Q^N_{\pi^N} [J_{0,a_1-b_0} =0]=
\bb Q^N_{\pi^N} [J_{b_0,a_1} =0]$. We proceed by induction to derive
\eqref{fc05}.

Let $\mf H_F$, $F\ge 1$, be the event $\{H (x_1) < H(A^N_{M,1})\} \cap
\mf D_F^c \cap \mf E_F^c$. We claim that
\begin{equation}
\label{fc09}
\begin{split}
& \sup_{y \in B(A_M^N, l_N)^c} \bigg| \bb Q^N_y \big[ \mf H_F\big]  -
\frac{ \bb Q^N_{\pi^N} \big[ J_{b_0,a_1} =1 \big]}
{\bb Q^N_{\pi^N} \big[ J_{b_0,a_1} \not =0 \big]} \bigg| \\
&\qquad\qquad\qquad\qquad\qquad
\le \; F^2 R_{L,N} \;+\; \bb Q^N_{\pi^N} \big[ J_{b_0,a_1} =0 \big]^F 
\;+\; \bb Q_y^N [\mf D_F]\;.
\end{split}
\end{equation}

Clearly, 
\begin{equation*}
\mf H_F \;=\; \mf D_F^c \cap \bigcup_{j=0}^{F-1} \Big\{  \bigcap_{k=0}^{j-1} 
\{J_{b_{k}, a_{k+1}} =0\} \cap \{J_{b_{j}, a_{j+1}} =1\} \Big\}\;.
\end{equation*}
Repeating the arguments which leaded to \eqref{fc05}, we get that for
each $0\le j\le F-1$ and any site $y$ in $B(A_M^N, l_N)^c$,
\begin{equation*}
\begin{split}
& \bb Q^N_y \Big[ \bigcap_{k=0}^{j-1} \{J_{b_{k}, a_{k+1}} =0\} 
\cap \{J_{b_{j}, a_{j+1}} =1\} \Big] \\
& \qquad =\; \bb Q^N_{\pi^N} \big[ J_{b_0,a_1} =0 \big]^{j} 
\bb Q^N_{\pi^N} \big[ J_{b_0,a_1} =1 \big] \;+\; R_N \;,
\end{split}
\end{equation*}
where the remainder $R_N$ is absolutely bounded by $(j+1)
R_{L,N}$. The value of the remainder $R_N$ may change from line to
line below. Summing over $j$, we get that
\begin{equation*}
\bb Q^N_y \big[ \mf H_F\big] \;=\; \bb Q^N_{\pi^N} \big[ J_{b_0,a_1}
=1 \big] \, \sum_{j=0}^{F-1} \bb Q^N_{\pi^N} \big [J_{b_0,a_1} =0
\big]^j \; +\; R_N \;,
\end{equation*}
where the remainder $R_N$ is now absolutely bounded by $F^2 R_{L,N}
+\bb Q_y^N[\mf D_F]$. This expression can be written as
\begin{equation*}
\frac{ \bb Q^N_{\pi^N} \big[ J_{b_0,a_1} =1 \big]}
{1-\bb Q^N_{\pi^N} \big[ J_{b_0,a_1} =0 \big]} \; +\; R_N
\end{equation*}
for a remainder $R_N$ absolutely bounded by $ F^2 R_{L,N}+\bb
Q^N_{\pi^N} \big[J_{b_0,a_1} =0 \big]^F +\bb Q_y^N[\mf D_F]$, which
is precisely \eqref{fc09}.

By the estimates \eqref{fc01}, \eqref{fc05}, \eqref{fc09}, the
supremum in \eqref{fc07} is bounded by
\begin{equation*}
2F \, T_{L,N} \;+\; 2S_{L,N} \;+\; 4 F^2 R_{L,N} \;+\; 
2\, \bb Q^N_{\pi^N} \big[ J_{b_0,a_1} =0 \big]^F
\end{equation*}
for every $F\ge 1$, $L\ge 1$, $\gamma >0$. By \eqref{fc11}, $\bb
Q^N_{\pi^N} [ J_{b_0,a_1} =0 ]$, which does not depend on $L$, is
bounded above by a constant strictly smaller than one. Hence, as
$N\uparrow\infty$, and then $L\uparrow\infty$, $R_{L,N}$, $S_{L,N}$
and $T_{L,N}$ vanish. It remains to let $F\uparrow\infty$ to
conclude the proof of \eqref{fc07}.

It remains to prove \eqref{fc08}. Decompose the event $\{J_{b_0,a_1} =
1\}$ according to the event that at least two sites in $A^N_M$ have
been visited in the time interval $[b_0,a_1]$ to get that
\begin{equation*}
\big| \mathbb{Q}^N_{\pi^N}[J_{b_0,a_1} = 1] - \mathbb{Q}^N_{\pi^N} [
Y_{[b_0,a_1]} \cap A^N_M = \{x_1\}] \, \big| \;\le\; 
\mathbb{Q}^N_{\pi^N} \big[\#(Y_{[b_0,a_1]} \cap A^N_M) > 1 \big]\;.
\end{equation*}
In this formula, $Y_{[b_0,a_1]}$ stands for the sites visited by the
random walk in the interval $[b_0,a_1]$, $Y_{[b_0,a_1]} = \{Y_t : b_0
\le t\le a_1\}$, and $\# A$ for the cardinality of $A$. By similar
reasons, 
\begin{equation*}
\big| \mathbb{Q}^N_{\pi^N} [ Y_{[b_0,a_1]} \cap A^N_M = \{x_1\}] 
- \mathbb{Q}^N_{\pi^N} [x_1 \in Y_{[b_0,a_1]}] \, \big| \;\le\; 
\mathbb{Q}^N_{\pi^N} \big[\#(Y_{[b_0,a_1]} \cap A^N_M) > 1 \big]\;.
\end{equation*}
Therefore,
\begin{equation*}
\big| \mathbb{Q}^N_{\pi^N}[J_{b_0,a_1} = 1]
- \mathbb{Q}^N_{\pi^N} [x_1 \in Y_{[b_0,a_1]}] \, \big| \;\le\; 
2 \mathbb{Q}^N_{\pi^N} \big[\#(Y_{[b_0,a_1]} \cap A^N_M) > 1 \big]\;.
\end{equation*}
Note that the probability $\mathbb{Q}^N_{\pi^N} [y \in Y_{[b_0,a_1]}]$
does not depend on $y$ by symmetry. In particular, it also follows
from the previous arguments that
\begin{equation*}
\big| \mathbb{Q}^N_{\pi^N} [J_{b_0,a_1} \neq 0]
- M \mathbb{Q}^N_{\pi^N} [x_1 \in Y_{[b_0,a_1]}] \, \big| \;\le\; 
(M+1) \mathbb{Q}^N_{\pi^N} \big[\#(Y_{[b_0,a_1]} \cap A^N_M) > 1 \big]
\end{equation*}
so that
\begin{equation*}
\big| \mathbb{Q}^N_{\pi^N}[J_{b_0,a_1} = 1] - (1/M)
\mathbb{Q}^N_{\pi^N} [J_{b_0,a_1} \neq 0] \, \big| \;\le\; 
4 \mathbb{Q}^N_{\pi^N} \big[\#(Y_{[b_0,a_1]} \cap A^N_M) > 1 \big]\;.
\end{equation*}
Equivalently,
\begin{equation*}
\bigg| \frac{\bb Q^N_{\pi^N} \big [J_{b_0,a_1} =1 \big]} 
{\bb Q^N_{\pi^N} \big [J_{b_0,a_1} \not = 0 \big]} - \frac 1M \bigg| 
\; \leq\;  \frac{4\, \mathbb{Q}^N_{\pi^N}
\big[ \#(Y_{[b_0,a_1]} \cap A^N_M) > 1 \big]}
{\bb Q^N_{\pi^N} \big [J_{b_0,a_1} \not = 0 \big]}\;\cdot
\end{equation*}
By the strong Markov property, the right hand side is bounded
above by
\begin{equation*}
\max_{1\le i\le M} \mathbb{Q}^N_{x_i} [H(A^N_{M,i}) < \gamma h^d_N]\;.
\end{equation*}
By \eqref{nohit2}, this expression vanishes as $N\uparrow\infty$ and
then $\gamma\downarrow 0$. This concludes the proof of
the lemma.
\end{proof}

We also need to estimate in dimension $2$ the probability that the
random walk escapes from a deep trap. More precisely,

\begin{lemma}
\label{aug6} 
Assume that $d = 2$ and let $\{l_N : N\ge 1\}$ be a sequence
satisfying \eqref{fc02}. Then, for any sequence $\{y^N \in
\mathbb{T}^2_N : N\ge 1\}$, 
\begin{equation*}
\lim_{N\to\infty} \log N\, \mathbb{Q}^N_{y^N} 
\big [ H(B(y^N,l_N)^c) < \hat \tau(y^N) \big] \;=\; \frac{\pi}{2}\;\cdot
\end{equation*}
\end{lemma}

\begin{proof} 
The number of visits that the random walk performs to $y^N$ before
exiting $B(y^N,l_N)$ is a geometric random variable, with failure
probability given by $\mathbb{Q}^N_{y^N}[H(B(y^N,l_N)^c) < \hat
\tau(y^N)]$. The inverse of this probability is equal to the expected
number of visits to $y^N$ before exiting $B(y^N,l_N)$. By
\cite{lawler}, Theorem~1.6.6, the expected number of visits, denoted
by $G_{B(y^N,l_N)}(y^N,y^N)$ in the notation of Green's functions, is
given by
\begin{equation*}
G_{B(y^N,l_N)}(y^N,y^N) \;=\; \frac{2}{\pi} \log N  \;+\; K \;+\;  
O(N^{-1})\;,
\end{equation*} 
for some constant $K$ in $\mathbb{R}$. The result follows from this
estimate.
\end{proof}

\begin{corollary}
\label{aug7} 
Assume that $d = 2$. Under the hypotheses of Lemma~\ref{aug5}, for $j
\neq 1$,
\begin{equation*}
\lim_{N\to\infty} \log N \, \mathbb{Q}^N_{x_j^N} \big[ H(x_1^N) <
\hat \tau(A^N_{M,1}) \big] \;=\;  \frac{\pi}{2M}\; \cdot 
\end{equation*}
\end{corollary}

\begin{proof}
This result follows from the strong Markov property and
Lemmas~\ref{aug5} and \ref{aug6}.
\end{proof}

\begin{corollary}
\label{aug8} 
Assume that $d = 2$ and let $\{R_N : N\ge 1\}$ be a sequence such that
$R_N \uparrow \infty$. Let $x^N, y^N \in \mathbb{T}^2_N$, such that
$d(x^N, y^N) \geq R_N$. Then,
\begin{equation}
\lim_{N\to\infty} \log N\, \Cap_{Y^N}(x^N,y^N) \;=\; \frac{\pi}{4}\;\cdot
\end{equation}
\end{corollary}

\begin{proof}
The corollary is a direct application of Lemma \ref{escape} and
Corollary~\ref{aug7}. 
\end{proof}

\subsection{Metastability of the trap model in dimension $d\ge 3$}

Recall that we denoted by $v_d$ the probability that a
nearest-neighbor, symmetric random walk on $\bb Z^d$ never returns to
its starting point. As in the previous subsections, we denote by
$Y^N_k$ the discrete time random walk on the torus $\bb T^d_N$,
inducing the law $\bb Q^N_x$ on $D(\bb Z_+, \bb T^d_N)$. The proof of
Theorem \ref{mt4} is divided in two parts. In Proposition \ref{s04b}
below we show that the trace process converges and in Corollary
\ref{s01b} that the time spent outside $A^N_M$ is negligible.

\begin{proposition}
\label{s04b}
Fix $M>1$ and $T>0$. As $N\uparrow\infty$, the process $\{ X^{N,M}_t :
0\le t\le T\}$ converges in distribution to the Markov process on
$\{1, \dots, M\}$ with generator $L_M$ given by
\begin{equation*}
(\mf L_M f) (i) \;=\; \frac {v_d}{M \hat w_i}
\sum_{j=1}^M [f(j) - f(i)]\; .
\end{equation*}
\end{proposition}

\begin{proof}
Fix $M\ge 1$ and denote by $r_{N,M} : \{1,\dots, M\} \times \{1,\dots,
M\} \to \bb R_+$ the jump rates of the trace process $\{X^{N,M}_t
: 0\le t\}$. By \eqref{f05b}, for $j\not = i$,
\begin{equation*}
r_{N,M}(i,j) \;=\; \frac 1{W^N_{x^N_i}}\,
\bb P^N_{x^N_i} \big [ H(x^N_j) < \tau (A^N_{M,j}) \big ] \; ,
\end{equation*}
where again $A^N_{M,j} = \{x^N_1, \dots , x^N_{j-1}, x^N_{j+1},
x^N_{M}\}$, $1\le j\le M \le N^d$. To compute this probability, we
need only to examine the discrete skeleton Markov chain:
\begin{equation*}
\bb P^N_{x^N_i} \Big [ H(x^N_j) < \tau (A^N_{M,j}) \Big] \;=\;
\bb Q^N_{x^N_i} \Big [ H(x^N_j) < \hat \tau (A^N_{M,j}) \Big] \;.
\end{equation*}
Since $x^N_j$ converges, as $N\uparrow\infty$, to $\hat x_j$, $1\le
j\le M$, and since $\min_{1\le i \not = j \le M} \Vert \hat x_i - \hat
x_j \Vert >0$, by Corollary \ref{aug1},
\begin{equation*}
\lim_{N\to\infty} \bb Q^N_{x^N_i} \Big [ H(x^N_j) < \hat \tau (A^N_{M,j})
\Big] \;=\; \frac {v_d} M\; \cdot
\end{equation*}
Hence, for $j\not = i$, $r_{N,M}(i,j)$ converges, as
$N\uparrow\infty$, to $(v_d/M \hat w_i)$, because $W^N_{x^N_i}$
converges to $\hat w_i$. This concludes the proof of the proposition.
\end{proof}

To examine the time spent by the random walk $\{X^N_t : 0\le t\le T\}$
on $\bb T^d_N \setminus A^N_M$, denote by $\Cap_N$ the capacity
associated to the process $X^N$. Of course, for any two disjoint
subsets $A$, $B$ of $\bb T^d_N$,
\begin{equation}
\label{f02b}
\Cap_{N} (A,B) \; =\; \frac{1}{W(\bb T^d)} \Cap_{Y^N} (A,B)\;.
\end{equation}

For $x \not =y$, in $\bb T^d_N$, denote by $\bb P^{N,x}_y$ the
probability measure on the path space $D(\bb R_+, \bb T^d_N \setminus
\{x\})$ induced by the trace of $\{X^N_t : t\ge 0\}$ on $\bb T^d_N
\setminus \{x\}$ starting from $y$. Expectation with respect to $\bb
P^{N,x}_y$ is denoted by $\bb E^{N,x}_y$.

\begin{lemma}
\label{s03b}
We have that
\begin{equation*}
\lim_{M\to\infty}  \, \limsup_{N\to\infty} \, \max_{1\le j\le M}
\, \max_{y : |y-x^N_j|=1} \, M\, \bb E^{N,x^N_j}_y \Big[ H (A^N_{M,j})
\Big] \;=\; 0\;.
\end{equation*}
\end{lemma}

\begin{proof}
Fix $1\le j\le M$ and $y\sim x^N_j$. By Lemma \ref{s02b}, the
expectation appearing in the statement of the lemma is equal to
\begin{equation}
\label{ft2}
\frac {1}{\Cap_N (y,A^N_{M,j})} \sum_{z\not = x^N_j} \nu^N (z)\,
\bb P^N_{z} \big[ H(y) < H(A^N_{M,j}) \big]\;.
\end{equation}

By \eqref{f02b}, the denominator is equal to
\begin{equation*}
\frac 1{W(\bb T^d)}\, \Cap_{Y^N}(y,A^N_{M,j}) \;\ge\;
\frac 1{W(\bb T^d)}\, \Cap_{Y^N}(y,x^N_1)\;.
\end{equation*}
In view of Corollary \ref{aug3}, this latter expression is bounded
below, uniformly in $N$, by a strictly positive constant.

To estimate the numerator in \eqref{ft2}, we need only to examine the
discrete skeleton Markov chain because $\bb P^N_{z} [ H(y) <
H(A^N_{M,j}) ] = \bb Q^N_{z} [ H(y) < H(A^N_{M,j}) ]$.  Fix a sequence
$\{ \ell_N : N\ge 1\}$ such that $1<\!\!< \ell_N <\!\!< N$ and let
$B_N = \{z\in \bb T^d_N : d(z,A^N_M) \le \ell_N\} \setminus A^N_M$,
$C_N = \{z\in \bb T^d_N : d(z,A^N_M) > \ell_N\}$.  Since $\bb Q^N_{z} [
H(y) < H(A^N_{M,j})]$ vanishes on the set $A^N_{M,j}$,
\begin{eqnarray*}
\sum_{z\not = x^N_j} \nu^N (z)\,
\bb Q^N_{z} \big[ H(y) < H(A^N_{M,j}) \big] &=&
\sum_{z\in B_N} \nu^N (z)\,
\bb Q^N_{z} \big[ H(y) < H(A^N_{M,j}) \big] \\
&+& \sum_{z\in C_N} \nu^N (z)\,
\bb Q^N_{z} \big[ H(y) < H(A^N_{M,j}) \big]\;.
\end{eqnarray*}
The first term on the right hand side is bounded by $\nu^N(B_N)$,
which vanishes as $N\uparrow\infty$ because $\ell_N <\!\!< N$. Since
$\ell_N >\!\!> 1$, by Lemma \ref{aug2}, as $N\uparrow\infty$, the
second term converges to $M^{-1} [1 - W(\bb T^d)^{-1} \sum_{1\le j\le
  M} \hat w_j]$. The expression inside brackets vanishes as
$M\uparrow\infty$ by definition of the sequence $\{\hat w_i\}$. This
proves the lemma.
\end{proof} 

Recall that $\Delta_{N,M} = \bb T^d_N\setminus A^N_M$.

\begin{corollary}
\label{s01b}
For every $t\ge 0$,
\begin{equation*}
\lim_{M\to\infty}  \, \limsup_{N\to\infty} \, \max_{1\le j\le M}
\, \bb E^{N}_{x^N_j} \big[ \mc T^{\Delta_{N,M}}_t \big] \;=\; 0\;.
\end{equation*}
\end{corollary}

\begin{proof}
Fix $M\ge 1$ and $1\le j\le M$. Consider the stochastic process $\hat
Z^{N,M}_t$ with state space $A^N_M$ defined as
\begin{equation*}
\hat Z^{N,M}_t\;=\; X^N (\sigma(t)) \;,
\end{equation*}
where $\sigma(t):=\sup\{s\le t : X^N_s\in A^N_M\}$. Hence, during an
excursion in $\Delta_{N,M}$ by $X^N$, the process $\hat Z^{N,M}_t$
stays at the last visited site in $A^N_M$.

For a path $\omega\in D(\bb R_+, \bb T^d_N)$ performing infinitely
many jumps, denote by $\tau_n(\omega)$, $n\ge 0$, the jumping times of
$\omega$: $\tau_0(\omega)=0$ and
\begin{equation*}
\tau_n(\omega) \;:=\; \inf\{t>\tau_{n-1}(\omega) :
\omega(t)\neq \omega(\tau_{n-1}(\omega)) \}\;.
\end{equation*}
Let
\begin{equation*}
T_n(\omega) \; := \; \tau_n(\omega)-\tau_{n-1}(\omega)\;,\quad n\ge 1\;,
\end{equation*}
and let $N_t$ be the number of jumps up to time $t$:
\begin{equation*}
N_t(\omega) \; := \; \sup\{j\ge 0 : \tau_j(\omega)\le t\}\;.
\end{equation*}

The process $\hat X_t^{N,M}$, defined on the path space $D(\bb R_+,
A^N_M)$, can be thought as a process on $D(\bb R_+, \bb T^d_N)$.
Couple the processes $\hat X_t^{N,M}$, $\hat Z_t^{N,M}$ forcing them to
visit the same sequence of sites. By Lemma \ref{s03b} and the proof of
Lemma 4.4 in \cite{bl1}, for every $K\ge 1$,
\begin{equation}
\label{f07b}
\lim_{M\to\infty} \limsup_{N\to\infty} \max_{1\le j\le M} \bb E^N_{x^N_j}
[\, \tau_{KM} (\hat Z^{N,M}) - \tau_{KM}(\hat X^{N,M})\,] \;=\; 0\;.
\end{equation}

Set $\hat N_t:=N_t(\hat Z^{N,M})$, $\hat T_n:=T_n(\hat Z^{N,M})$ and
$T_n:=T_n(\hat X^{N,M})$. Fix $1\le j\le M$. Under $\bb P^N_{x^N_j}$,
\begin{eqnarray*}
\mc T^{\Delta_{N,M}}_t \;\le\; t \land \sum_{n=1}^{\hat N_t +1}
(\hat T_n - T_n)
\;\le\; \mathbf 1\{\hat N_t \ge K M\} \; t \;+\;
\sum_{n=1}^{KM}(\hat T_n - T_n)
\end{eqnarray*}
for any positive integer $K$. Therefore,
\begin{equation*}
\bb E^N_{x^N_j}[\mc T^{\Delta_{N,M}}_t] \; \le \; t \;
\bb P^N_{x^N_j}[\, \hat N_t  \ge KM \,] \;+\;
\bb E^N_{x^N_j} \big[\,\tau_{KM}(\hat Z^{N,M}) - \tau_{KM}(\hat X^{N,M})\,
\big]\;.
\end{equation*}

By \eqref{f07b}, the second term vanishes as $N\uparrow\infty$ and then
$M\uparrow\infty$. It remains to prove that
\begin{equation}
\label{f08b}
\lim_{K\to\infty} \limsup_{M\to\infty} \limsup_{N\to\infty}
\max_{1\le j\le M} \bb P^N_{x^N_j}[\,N_t(\hat Z^{N,M}) \ge KM \,]
\;=\;0\;.
\end{equation}
Since $N_t(\hat Z^{N,M})\le N_t(\hat X^{N,M})$, $\bb P^N_{x^N_j}$ - a.s.,
\begin{equation*}
\bb P^N_{x^N_j}[\,N_t(\hat Z^{N,M}) \ge KM \,] \;\le\;
\bb P^N_{x^N_j}[\,N_t(\hat X^{N,M}) \ge KM \,]\;.
\end{equation*}
Fix $M\ge 1$ and $1\le j\le M$ such that
\begin{eqnarray*}
\limsup_{N\to\infty} \max_{1\le k\le M}
\bb P^N_{x^N_k}[\,N_t(\hat X^{N,M})\ge KM \,]
\;=\; \limsup_{N\to\infty} \bb P^N_{x^N_j}[\,N_t(\hat X^{N,M})\ge KM
\,] \; .
\end{eqnarray*}
Since $[N_t\ge KM]$ is a closed set for the Skorohod topology on
$D(\bb R_+, A^N_M)$, since $N_t (\hat X^{N,M})$ has the same
distribution as $N_t (X^{N,M})$ and since, by Proposition \ref{s04b},
$X^{N,M}$ converges in distribution to $Z^M$,
\begin{equation*}
\limsup_{N\to \infty}\bb P^N_{x^N_j} [N_t (\hat X^{N,M}) \ge KM] \;\le\;
\max_{1\le k\le M} P_{k} [N_t(Z^M) \ge KM]\;,
\end{equation*}
where $P_k$ is the distribution of the process $Z^M$ starting from $k$.

To estimate the right hand side, we compare $N_t(Z^M)$ with a counting
process $C_t$ in which we replace the holding times $T_n$ by $0$ if
$Z^M(\tau_{n-1}) \not = 1$. In other words, let $G_0:=C_0$ be the
number of times the process $Z^M$ jumped before hitting $1$ for the
first time. Since $Z^M$ jumps from any site uniformly to all others,
$G_0$ is a random variables with geometric distribution: $P[G_0 = n] =
(1/M) [(M-1)/M]^{n-1}$, $n\ge 1$. When hitting $1$, as $Z^M$, the
process $C_t$ stays there for a mean $\hat w_1/v_d$ exponential time.
At the end of this exponential time, $C_t$ jumps from $G_0$ to
$G_0+G_1$, where $G_1$ stands for the number of jumps performed by
$Z^M$ before hitting $1$ again.

By construction, $N_t(Z^M) \le C_t$ for all $t\ge 0$ and $C_t =
\sum_{0\le j \le \hat N_t} G_j$, where $\{G_j : j\ge 0\}$ are i.i.d.\!
random variables with geometric distribution: $P[G_1 = n] = (1/M)
[(M-1)/M]^{n-1}$, $n\ge 1$; and $\hat N_t$ is a Poisson process with
rate $v_d/\hat w_1$, independent of the sequence $\{G_i\}$. In
particular, $E_{P_k} [N_t(Z^M)] \le M [1+ (t \hat w_1/v_d)]$. This
proves \eqref{f08b} and the corollary.
\end{proof}

\begin{proof}[Proof of Theorem \ref{mt4}]
Theorem \ref{mt4} follows from Proposition \ref{s04b} and Corollary
\ref{s01b}.
\end{proof}

\begin{proof}[Proof of Theorem \ref{mt2}]
Denote by $\{Z_t : t\ge 0\}$ the $K$-process with parameters $\{\hat
w_i/v_d : i\ge 1\}$ and $c=0$. Using independent exponential and
Poisson random variables, we may define in the same probability space
$(\Omega, \mc F, P)$ processes $\{\bb X^{N,M}_t : t\ge 0\}$, $\{\bb
Z^{M}_t : t\ge 0\}$ and $\{\bb Z_t : t\ge 0\}$ which have the same
distribution as $\{X^{N,M}_t : t\ge 0\}$, $\{Z^{M}_t : t\ge 0\}$ and
$\{Z_t : t\ge 0\}$, respectively. Fix a common starting point $j$ for
all processes and $T>0$.  By \cite[Lemma 3.11]{fm1}, $\{\bb Z^{M}_t :
0\le t\le T\}$ converges a.s., as $M\uparrow\infty$, to $\{\bb Z_t :
t\ge 0\}$ in the Skorohod metric. On the other hand, by Proposition
\ref{s04b}, if $d_S$ stands for the Skorohod metric on $D([0,T],
\overline{\bb N})$, for every $M\ge 1$ and $\epsilon >0$,
\begin{equation*}
\lim_{N\to\infty}
P\big[ d_S(\bb X^{N,M}, \bb Z^M) > \epsilon\big] \;=\; 0\; .
\end{equation*}
In particular, there exists a strictly increasing sequence $\{N^*_M :
M\ge 1\}$, such that
\begin{equation*}
P\big[ d_S(\bb X^{N,M}, \bb Z^M) > M^{-1} \big] \;\le\; \frac 1M
\end{equation*}
for all $N\ge N^*_M$. Hence, by the triangular inequality, for any
sequence $\{N_M : M\ge 1\}$ such that $N_M\ge N^*_M$,
\begin{equation*}
\lim_{M\to\infty}
P\big[ d_S(\bb X^{N_M,M}, \bb Z) > \epsilon\big] \;=\; 0\; .
\end{equation*}
for any $\epsilon >0$. The sequence $\{\ell^*_N: N\ge 1\}$, defined as
the inverse of $\{N^*_M : M\ge 1\}$, fulfills the requirements of the
first part of Theorem \ref{mt2}.

To prove the second statement of Theorem \ref{mt2}, fix $M\ge 2$ and
observe that $\mc T^{\Delta_{N,\ell_N}}_t \le H(A^N_M) + \mc
T^{\Delta_{N,M}}_t \circ \theta (H(A^N_M))$ provided $\ell_N\ge M$.
In this formula, $\theta(s) : D(\bb R_+, \bb T^d_N) \to D(\bb R_+, \bb
T^d_N)$, $s\ge 0$, stands for the time shift by $s$ of a path
$\omega$: $(\theta(s) \omega)(t) = \omega(s+t)$, $t\ge 0$. Therefore,
\begin{equation*}
\max_{1\le j\le \ell_N} \bb E^N_{x^N_j}
\big[ \mc T^{\Delta_{N,\ell_N}}_t \big] \;\le\;
\max_{1\le j\le \ell_N} \bb E^N_{x^N_j} \big[ H(A^N_M) \big]
\;+\; \max_{1\le j\le M} \bb E^N_{x^N_j}
\big[ \mc T^{\Delta_{N,M}}_t \big]
\end{equation*}
provided $M\le \ell_N$. By Corollary \ref{s01b}, the second expression
converges to $0$ as $N\uparrow\infty$, $M\uparrow\infty$. By
definition of $H(A^N_M)$, the first expectation on the right hand side
vanishes for $1\le j\le M$. For $M<j\le \ell_N$, by \eqref{f06} with
$F=E=\bb T^d_N$,
\begin{equation*}
\bb E^N_{x^N_j} \big[ H(A^N_M) \big] \;\le\;
\frac{\nu^N(\Delta_{N,M})}{\Cap (x^N_j, A^N_M)}\;\cdot
\end{equation*}
The denominator is bounded below by $\min_{z\in \bb T^d_N} \max_{1\le
  k\le M} \Cap(z,x^N_k)$. Since $\Vert \hat x_1 - \hat x_2\Vert >0$,
by Corollary \ref{aug3}, this expression is bounded below by a
positive constant, uniformly in $M$ and $N$. This proves the second
statement of Theorem \ref{mt2} because $\nu^N(\Delta_{N,M})$ vanishes
as $N\uparrow\infty$, $M\uparrow\infty$.
\end{proof}

\subsection{Metastability for $d = 2$.} 

In this subsection, we adapt to dimension $2$ the results presented in
the previous subsection. Most of the proofs are similar to the case $d
\geqslant 3$.

The main difference with respect to dimension $3$ is that the process
is speeded up by $\log N$. Recall from Section \ref{sec1} that we
denote by $\{\mf X^N_t : t\ge 0\}$ the random walk with generator $\mc
L_N$, defined in \eqref{fc03}, speeded up by $\log N$. Moreover, $\mb
P^N_x$, $\bb P^N_x$, $x\in \bb T^2_N$, stand for the probability
measure on $\bb D(\bb R_+, \bb T^2_N)$ induced by the processes $\{\mf
X^N_t : t\ge 0\}$, $\{X^N_t : t\ge 0\}$ starting from $x$.

\begin{proposition} 
\label{s14}
For a fixed $M > 1$, and $T > 0$, the process $\{\mf X_{t}^{N,M}
\,:\, 0 \leqslant t \leqslant T \}$ converges in distribution, as
$N\uparrow\infty$, to the Markov process in $\{1,\dots,M\}$ given by
the following generator
\begin{equation*}
(\mf L^\star_Mf)(i)= \frac{\pi}{2} \frac{1}{M \hat w_{i}} 
\sum_{j = 1}^M [f(j) - f(i)]\; . 
\end{equation*}
\end{proposition}

\begin{proof}
As in the proof of Proposition~\ref{s04b}, we use (\ref{f05b}) to
write the jump rates of $\mf X_{t}^{N,M}$ in terms of the excursion
probabilities between the very deep traps:
\begin{equation*}
r_{N,M}(i,j) \;=\; \frac{\log N}{W^N_{x^N_i}}
\mathbb{Q}^N_{x_i^N} \big[ H(x^N_j) < \hat \tau(A^N_{M,j}) \big] \;. 
\end{equation*}
Note the factor $\log N$ which appears because the generator $\mc L_M$
is multiplied by this constant.

Consider a sequence $\{ l_N : N\ge 1\}$ satisfying \eqref{fc02}.
Use the strong Markov property on $H(B(x^N_i,l_N)^c)$ to obtain
\begin{equation*}
\begin{split}
& \mathbb{Q}^N_{x_i^N} \big[ H(x^N_j) < \hat \tau(A^N_{M,j}) \big] \\
&\quad \;=\; \mathbb{Q}^N_{x_i^N} \Big[ \mb 1\big\{ 
H(B(x^N_i,l_N)^c) < \hat \tau(x^N_i) \big\} \,
\mathbb{Q}^N_{Y(H(B(x^N_i,l_N)^c))} \big[ H(x^N_j) < H(A^N_{M,j}) \big]
\, \Big] \;. 
\end{split}
\end{equation*}
Therefore,
\begin{equation*}
\begin{split}
& \Big| \log N\, \mathbb{Q}^N_{x_i^N} \big[ H(x^N_j) < 
\hat \tau(A^N_{M,j}) \big] \,-\, \frac{\pi}{2M} \Big|  \\
&\qquad\qquad\quad \;\le\;
\Big| \log N\, \mathbb{Q}^N_{x_i^N} \Big[ H(B(x^N_i,l_N)^c) < \hat
\tau(x^N_i) \big] \;-\; \frac \pi 2 \Big| \\
& \qquad\qquad\quad \;+\;
\frac \pi 2 \, \sup_{z \in \partial B(x^N_i,l_N)} \Big| \mathbb{Q}^N_{z} 
\big[ H(x^N_j) < H(A^N_{M,j}) \big]  - \frac 1M \Big| \;. 
\end{split}
\end{equation*}
By Lemmas~\ref{aug5} and \ref{aug6}, these expressions vanish as
$N\uparrow\infty$. Since $W^N_{x^N_i}$ converges towards $\hat w_i$,
$1\le i\le M$, we are done.
\end{proof}

Recall the definition of the measures $\bb P^{N,x}_y$, $x\not = y \in
\bb T^2_N$, introduced in the previous subsection. It corresponds to
the trace on $\bb T^2_N \setminus \{x\}$ of the process $\{X^N_t :
t\ge 0\}$, which has not been speeded up.

\begin{lemma}
\label{s12}
In dimension $2$, 
\begin{equation*}
\lim_{M\to\infty}  \, \limsup_{N\to\infty} \, \max_{1\le j\le M}
\, \max_{y : |y-x^N_j|=1} \, \frac{M}{\log N} \, \bb 
E^{N,x^N_j}_y \Big[ H (A^N_{M,j}) \Big] \;=\; 0\;.
\end{equation*}
\end{lemma}

\begin{proof}
The proof of this result follows the same argument as in
Lemma~\ref{s03b}. One only notes that the denominator of (\ref{ft2})
is now multiplied by $\log N$ which allows us to use
Corollary~\ref{aug8} in place of Corollary~\ref{aug3}. The argument to
bound the numerator is also the same. However, one should choose a
sequence $\{\ell_N : N\ge 1\}$ satisfying \eqref{fc02} in order to
apply Lemma~\ref{aug5}.
\end{proof}

For $x \not =y$, in $\bb T^2_N$, denote by $\mb P^{N,x}_y$ the
probability measure on the path space $D(\bb R_+, \bb T^d_N \setminus
\{x\})$ induced by the trace of $\{\mf X^N_t : t\ge 0\}$ on $\bb T^d_N
\setminus \{x\}$ starting from $y$. Expectation with respect to $\mb
P^{N,x}_y$ is denoted by $\mb E^{N,x}_y$. 
The difference between $\mb P^{N,x}_y$ and $\bb P^{N,x}_y$ is that the
first probability measure is associated to the random walk speeded up
by $\log N$. Therefore, for every subset $A$ of $\bb T^2_N \setminus
\{x^N_j\}$,
\begin{equation*}
\mb E^{N,x^N_j}_y \big[ H (A) \big] \;=\; \frac 1{\log N} 
\, \bb E^{N,x^N_j}_y \big[ H (A) \big]\; .
\end{equation*}
In particular, it follows from the previous lemma that
\begin{equation}
\label{fc04}
\lim_{M\to\infty}  \, \limsup_{N\to\infty} \, \max_{1\le j\le M}
\, \max_{y : |y-x^N_j|=1} \, M \, \mb
E^{N,x^N_j}_y \big[ H (A^N_{M,j}) \big] \;=\; 0\;.
\end{equation}

\begin{corollary}
\label{s13}
In dimension $2$, for every $t\ge 0$,
\begin{equation*}
\lim_{M\to\infty}  \, \limsup_{N\to\infty} \, \max_{1\le j\le M}
\, \mb E^{N}_{x^N_j} \big[ \mc T^{\Delta_{N,M}}_t \big] \;=\; 0\;.
\end{equation*}
\end{corollary}

\begin{proof}
  The argument is identical to the one in $d \geqslant 3$ presented in
  Corollary~\ref{s01b}. We just use \eqref{fc04} and Proposition
  \ref{s14} instead of Lemma \ref{s03b} and Proposition \ref{s04b}.
  At the end of the proof, the rate of the process $N_t(Z^M)$ is
  replaced by $\pi/(2\hat w_1)$, but its exact value is superfluous.
\end{proof}

\begin{proof}[Proof of Theorem \ref{mt5}]
  The proof is a direct consequence of Proposition~\ref{s14} and
  Corollary~\ref{s13}.
\end{proof}

\subsection{Dimension $2$ with no acceleration}

We prove in this subsection that in dimension $2$ the trap model with
generator \eqref{fc03} starting from a very deep trap does not
move. Hence, on the order $1$ scale, the random walk does not move
and on the scale $\log N$ it converges to the $K$-process in which
all the geometry is wiped out.

\begin{proposition}
\label{st1}
For every $j\ge 1$, every $t>0$ and every sequence $\{\ell_N : N\ge
1\}$ such that $\ell_N\uparrow\infty$,
\begin{equation*}
\lim_{N\to\infty} \bb P^N_{x^N_j} \big[ \, \big| X^N(t) - x^N_j
\big| \ge \ell_N \big]  \;=\; 0\;.
\end{equation*}
\end{proposition}

\begin{proof}
Fix $j\ge 1$ and a sequence $\{\ell_N : N\ge 1\}$ such that
$\ell_N\uparrow\infty$. Following \cite{bc1}, denote by $S_N : \bb
Z_+\to \bb R$ the clock process: $S_N(0)=0$,
\begin{equation*}
S_N(k) \;=\; \sum_{i=0}^{k-1} \mf e_i \, W^N(Y^N(i))\; , \quad k\ge 1\;,
\end{equation*}
where $\{Y^N(i) : i\ge 0\}$ is a nearest-neighbor, symmetric, discrete
time random walk on $\bb T^2_N$ starting from $x^N_j$; $\{\mf e_i :
i\ge 0\}$ is a sequence of i.i.d.\! mean one, exponential random
variables, independent from the Markov chain $\{Y^N(i)\}$; and $W^N(x)
= W^N_x$, $x\in \bb T^2_N$. Denote by $T_N:\bb R_+ \to \bb R_+$ the
inverse of $S_N$:
\begin{equation*}
T_N(t) \;=\; \sup \big\{ k : S_N(k) \le t\big\}\;.
\end{equation*}

Clearly $\{X^N(t) : t\ge 0\}$ has the same distribution as
$\{Y^N(T_N(t)) : t\ge 0\}$. Hence,
\begin{equation}
\label{ft01}
\begin{split}
& \bb P^N_{x^N_j} \big[ \, \big| X^N(t) - x^N_j \big| \ge \ell_N \big]
\;=\; P^N \big[ \, \big| Y^N(T_N(t)) - x^N_j \big| \ge \ell_N \big] \\
& \qquad \le\; P^N \big[ \, \max _{0\le k\le r_N} \big| Y^N(k) - x^N_j \big|
\ge \ell_N \big] \;+\; P^N \big[ \, T_N(t) \ge r_N \big]\;,
\end{split}
\end{equation}
for a sequence $r_N$ such that $1<\!\!< r_N <\!\!< \ell_N^2$.

We estimate separately the expressions on the right hand side of
\eqref{ft01}. Since $T_N$ is the inverse of $S_N$, $\{T_N(t) \ge r_N
\} = \{S_N(r_N) \le t\}$. In particular,
\begin{equation*}
P^N \big[ \, T_N(t) \ge r_N \big] \;\le\;
P^N \big[ \hat w_j \sum_{i=0}^{r_N-1} \mf e_i  \, \mb 1\{Y^N(i)
= x^N_j\} \le t \big]\;,
\end{equation*}
because $S_N(k) \ge W^N_{x^N_j} \sum_{0\le i\le k-1} \mf e_i \, \mb
1\{Y^N(i) = x^N_j\}$, $W^N_{x^N_j} \ge \hat w_j$. Since $\{Y^N(k) :
k\ge 0\}$ starts from $x^N_j$ and the two-dimensional random walk is
recurrent, the previous probability vanishes as $N\uparrow\infty$
because $r_N\uparrow\infty$.

On the other hand, since $Y^N(k) - x^N_j$ is a bi-dimensional
martingale, by Doob's inequality,
\begin{equation*}
P^N \big[ \, \max _{0\le k\le r_N} \big| Y^N(k) - x^N_j \big|
\ge \ell_N \big] \;\le\; \frac{4 r_N}{\ell_N^2}\; ,
\end{equation*}
which vanishes as $N\uparrow\infty$. This concludes the proof of the
proposition.
\end{proof}

\medskip
\noindent{\bf Acknowledgments.} We thank David Windisch and G. Ben
Arous for fruitful discussions and Stefano Olla for indicating
reference \cite{pv1}.

\end{document}